\documentclass[twoside]{amsart}
\usepackage{amsmath,amsfonts,amsthm,mathrsfs}
\usepackage[unicode]{hyperref}

\usepackage[alphabetic,initials]{amsrefs}

\usepackage{amssymb}
\usepackage{enumitem}

\usepackage[dpi=1270,PostScript=dvips,heads=LaTeX]{diagrams}
\diagramstyle{small,labelstyle=\scriptstyle}
\newcommand{\rToDots}{\rTo[l>=3.3em]~{\raisebox{0.1pt}{$\,\dotsb\,$}}}
\newarrow{Equal}   =====
\newarrow{IntoBold}   {boldhook}---> 

\usepackage{chngcntr}

\makeatletter
\let\@@seccntformat\@seccntformat
\renewcommand*{\@seccntformat}[1]{%
  \expandafter\ifx\csname @seccntformat@#1\endcsname\relax
    \expandafter\@@seccntformat
  \else
    \expandafter
      \csname @seccntformat@#1\expandafter\endcsname
  \fi
    {#1}%
}
\newcommand*{\@seccntformat@subsection}[1]{%
  \textbf{\csname the#1\endcsname.}
}
\makeatother

\makeatletter
\let\@paragraph\paragraph
\renewcommand*{\paragraph}[1]{%
	\vspace{0.3\baselineskip}%
	\@paragraph{\textit{#1}}%
}
\makeatother

\counterwithin{equation}{subsection}
\counterwithout{subsection}{section}
\counterwithin{figure}{subsection}

\newcommand{\subsecref}[1]{\S\ref{#1}}

\newtheorem{theorem}[equation]{Theorem}
\newtheorem*{theorem*}{Theorem}
\newtheorem{lemma}[equation]{Lemma}
\newtheorem*{lemma*}{Lemma}
\newtheorem{corollary}[equation]{Corollary}
\newtheorem{proposition}[equation]{Proposition}
\newtheorem*{proposition*}{Proposition}

\theoremstyle{definition}

\newtheorem*{definition*}{Definition}
\theoremstyle{remark}

\newtheorem{example}[equation]{Example}
\newtheorem*{example*}{Example}
\newtheorem*{note}{Note}

\theoremstyle{plain}

\makeatletter
\let\old@caption\caption
\renewcommand*{\caption}[1]{%
	\setcounter{figure}{\value{equation}}%
	\stepcounter{equation}%
	\old@caption{#1}\relax%
}
\makeatother

\newcommand{\MHM}{\operatorname{MHM}}

\newcommand{\pt}{\mathit{pt}}
\newcommand{\Dmod}{\mathcal{D}}
\newcommand{\Mmod}{\mathcal{M}}
\newcommand{\Nmod}{\mathcal{N}}

\newcommand{\derR}{\mathbb{R}}

\newcommand{\decal}[1]{\lbrack #1 \rbrack}

\newcommand{\ltriangle}[4][]%
{\begin{diagram}[#1]%
	{#2} &\rTo& {#3} &\rTo& {#4} &\rTo& {#2 \decal{1}}%
\end{diagram}}

\newcommand{\shH}{\shf{H}}

\newcommand{\contr}[2]{#1 \lrcorner\, #2}

\newcommand{\pr}{\mathit{pr}}

\newcommand{\abs}[1]{\lvert #1 \rvert}

\newcommand{\eps}{\varepsilon}

\newcommand{\tensor}{\otimes}

\newcommand{\shHom}{\mathscr{H}\hspace{-2.7pt}\mathit{om}}
\newcommand{\shExt}{\mathscr{E}\hspace{-1.5pt}\mathit{xt}}
\newcommand{\Hom}{\operatorname{Hom}}
\newcommand{\Ext}{\operatorname{Ext}}

\newcommand{\ZZ}{\mathbb{Z}}
\newcommand{\QQ}{\mathbb{Q}}

\newcommand{\CC}{\mathbb{C}}
\newcommand{\HH}{\mathbb{H}}

\newcommand{\PPn}[1]{\mathbb{P}^{#1}}

\newcommand{\menge}[2]{\bigl\{ \thinspace #1 \thinspace\thinspace \big\vert%
\thinspace\thinspace #2 \thinspace \bigr\}}

\DeclareMathOperator{\Res}{Res}

\DeclareMathOperator{\Proj}{Proj}

\DeclareMathOperator{\depth}{depth}
\DeclareMathOperator{\rat}{rat}

\DeclareMathOperator{\Sym}{Sym}
\DeclareMathOperator{\Gr}{Gr}
\DeclareMathOperator{\DR}{DR}

\newcommand{\Ch}{\mathit{SS}}

\newcommand{\define}[1]{\emph{#1}}

\newcommand{\lie}[2]{\lbrack #1, #2 \rbrack}

\newcommand{\shf}[1]{\mathscr{#1}}
\newcommand{\OX}{\shf{O}_X}
\newcommand{\OmX}[1]{\Omega_X^{#1}}

\newcommand{\prim}{_\mathit{prm}}

\newcommand{\restr}[1]{\big\vert_{#1}}

\newcommand{\argbl}{-}

\def\overbar#1#2#3{{%
	\setbox0=\hbox{\displaystyle{#1}}%
	\dimen0=\wd0
	\advance\dimen0 by -#2 
	\vbox {\nointerlineskip \moveright #3 \vbox{\hrule height 0.3pt width \dimen0}%
		\nointerlineskip \vskip 1.5pt \box0}%
}}

\newcommand{\into}{\hookrightarrow}

\newcommand{\gl}{\mathfrak{gl}}

\newcommand{\HC}{H_{\CC}}
\newcommand{\shHO}{\shf{H}_{\shf{O}}}

\newcommand{\famX}{\mathscr{X}}
\newcommand{\pil}{\pi_{\ast}}
\newcommand{\Hvan}[1]{H_{\mathrm{van}}^{#1}}
\newcommand{\jl}{j_{\ast}}
\newcommand{\jlreg}{j_{\ast}^{\mathrm{reg}}}

\newcommand{\jlsl}{j_{!\ast}}
\newcommand{\ju}{j^{\ast}}

\newcommand{\fl}{f_{\ast}}

\newcommand{\iu}{i^{\ast}}
\newcommand{\ius}{i^{!}}
\newcommand{\il}{i_{\ast}}

\newcommand{\shF}{\shf{F}}
\newcommand{\shG}{\shf{G}}
\newcommand{\shE}{\shf{E}}

\newcommand{\shO}{\shf{O}}

\usepackage{ifthen}

\newcommand{\disttrianglelinelong}[8]{%
\begin{diagram}[midshaft,#1]%
	{#2} &\rTo^{#6}& {#3} &\rTo^{#7}& {#4} &\rTo^{#8}& {#5}%
\end{diagram}%
}

\DeclareFontFamily{OT1}{pzc}{}
\DeclareFontShape{OT1}{pzc}{m}{it}{<-> s * [1.05] pzcmi8t}{}
\DeclareMathAlphabet{\mathpzc}{OT1}{pzc}{m}{it}

\newarrow{Equal}   =====
\newarrow{IntoBold}   {boldhook}--->

\newcommand{\famY}{\mathscr{Y}}
\newcommand{\dualX}{X^{\vee}}
\newcommand{\dualV}{V^{\vee}}

\newcommand{\Psm}{P^{\mathit{sm}}}
\newcommand{\OPsm}{\shf{O}_{\Psm}}
\newcommand{\famXsm}{\famX^{\mathit{sm}}}
\newcommand{\pism}{\pi^{\mathit{sm}}}
\newcommand{\pisml}{\pi_{\ast}^{\mathit{sm}}}

\newcommand{\OPX}{\shf{O}_{P \times X}}

\newcommand{\OmPXP}[1]{\Omega_{P \times X / P}^{#1}}

\newcommand{\OmP}[1]{\Omega_P^{#1}}
\newcommand{\OfamX}{\shf{O}_{\famX}}
\newcommand{\OfamY}{\shf{O}_{\famY}}
\newcommand{\OP}{\shf{O}_P}
\newcommand{\OQ}{\shf{O}_Q}
\newcommand{\OPQ}{\shf{O}_{P \times Q}}
\newcommand{\shC}{\shf{C}}
\newcommand{\dPXP}{d_{P \times X / P}}

\newcommand{\DP}{\mathcal{D}_P}
\newcommand{\Vvan}{V_{\mathit{ev}}^{n-1}}
\newcommand{\Mvan}{\Mmod_{\mathit{ev}}}

\renewcommand{\pr}{\mathit{pr}}
\newcommand{\prP}{\pr_{\!P}}
\newcommand{\prX}{\pr_{\!X}}

\newcommand{\prPl}{\pr_{\!P\ast}}
\newcommand{\prPu}{\pr_{\!P}^{\ast}}

\newcommand{\prQu}{\pr_{\!Q}^{\ast}}

\newcommand{\prXu}{\pr_{\!X}^{\ast}}
\newcommand{\ql}{q_{\ast}}
\newcommand{\gu}{g^{\ast}}

\newcommand{\phiu}{\phi^{\ast}}

\newcommand{\psiu}{\psi^{\ast}}

\newcommand{\aPu}{a_P^{\ast}}
\newcommand{\aPl}{a_{P\ast}}
\newcommand{\aXu}{a_X^{\ast}}
\newcommand{\afamXl}{a_{\famX\ast}}

\newcommand{\aXl}{a_{X\ast}}

\newcommand{\QXH}{\QQ_X^H}
\newcommand{\QPXH}{\QQ_{P \times X}^H}
\newcommand{\QfamXH}{\QQ_{\famX}^H}
\newcommand{\famU}{\mathscr{U}}
\newcommand{\QUH}{\QQ_{\famU}^H}
\newcommand{\QPH}{\QQ_P^H}

\newcommand{\dfamX}{d_{\famX}}
\newcommand{\dPX}{d_{P \times X}}
\newcommand{\dfamU}{d_{\famU}}

\newcommand{\Pervalg}[1]{\operatorname{Perv}_{\mathit{alg}}(#1)}

\newcommand{\Rmod}{\mathcal{R}}

\newcommand{\Emod}{\mathcal{E}}

\renewcommand{\restr}[1]{\vert_{#1}}
\newcommand{\ResfamXP}{\Res_{\famX / P}}

\newcommand{\NXQ}{N_{X \mid Q}}

\newcommand{\primeE}{{'\!E}}
\newcommand{\pprimeE}{{''\!E}}
\newcommand{\sfamX}{s_{\famX}}
\newcommand{\alfamX}{\alpha_{\famX}}

\newcommand{\cohshf}{\mathcal{H}}

\newcommand{\PP}{\mathbb{P}}
\newcommand{\Hprim}[1]{H_0^{#1}}
\newcommand{\HXprim}{\Hprim{n}(X)}

\renewcommand{\gl}{g_{\ast}}

\newcommand{\Rvan}[1]{R_{\mathit{ev}}^{#1}}
\newcommand{\famS}{\mathfrak{S}}
\newcommand{\famI}{\mathscr{Q}}

\renewcommand{\Hvan}[1]{H_{\mathit{ev}}^{#1}}

\newcounter{thmA}
\newtheorem{theoremA}[thmA]{Theorem}

\begin{document}

\title[Residues and D-modules]{Residues and filtered D-modules}
\author[C.~Schnell]{Christian Schnell}
\address{Department of Mathematics, Statistics \& Computer Science \\
University of Illinois at Chicago \\
851 South Morgan Street \\
Chicago, IL 60607}
\email{cschnell@math.uic.edu}
\subjclass[2000]{14D07; 32C38, 14F10}
\keywords{Residues, Hypersurfaces, Filtered D-module, Mixed Hodge module}
\begin{abstract}
For an embedding of sufficiently high degree of a smooth projective variety $X$ into
projective space, we use residues to define a filtered holonomic $\Dmod$-module
$(\Mmod, F)$ on the dual projective space. This gives a concrete description of the
intermediate extension to a Hodge module on $P$ of the variation of Hodge
structure on the middle-dimensional cohomology of the hyperplane sections of $X$. We
also establish many results about the sheaves $F_k \Mmod$, such as positivity,
vanishing theorems, or reflexivity.
\end{abstract}
\maketitle

\section{Overview}

\subsection{Introduction}
\label{subsec:intro}

Let $X$ be a complex projective manifold of dimension $n$, and let $D \subseteq X$ be
a smooth and very ample hypersurface. The cohomology of the complement $X \setminus
D$ can be represented by meromorphic forms on $X$ with poles along $D$. It also
carries a mixed Hodge structure, and Griffiths \cite{RatInt} (for $X = \PPn{n}$) and
Green \cite{GreenPeriod} (in general) have shown that the Hodge filtration is
basically the filtration by pole order, provided that the line bundle $\OX(D)$ is
sufficiently ample. One consequence is a formula for the Hodge filtration on the
vanishing cohomology $\Hvan{n-1}(D, \QQ)$ of the hypersurface (see
\subsecref{subsec:summary} below).

The purpose of this article is to generalize the above picture to the case when $D$
is allowed to have singularities. Fix a very ample line bundle $\OX(1)$ on $X$, and
consider all hypersurfaces, smooth or singular, that belong to the linear system $P =
\abs{\OX(1)}$. Using residues, we construct a $\Dmod$-module $\Mmod$ on the
projective space $P$, together with a good filtration $F = F_{\bullet} \Mmod$ by
$\OP$-coherent subsheaves. Provided that $\OX(1)$ is sufficiently ample, we then show
that $(\Mmod, F)$ is regular and holonomic, and in fact underlies a polarized Hodge
module on $P$. We also show that the coherent sheaves $F_k \Mmod$ have many
remarkable properties. The $\Dmod$-module $\Mmod$ is intimately related to the
geometry of the incidence variety $\famX \subseteq P \times X$, and in the process of
describing $\Mmod$, we recover the results about $\famX$ obtained by Brosnan, Fang,
Nie, and Pearlstein in \cite{BFNP}. We also strengthen one of their theorems, by
showing that, in the Decomposition Theorem, no terms with support in the dual variety
appear once the vanishing cohomology of the hypersurfaces is nontrivial.

These results seem interesting for two reasons. One is theoretical: Computing minimal
extensions of holonomic $\Dmod$-modules is a difficult problem, except in
the case of a divisor with normal crossing singularities. In the situation above, we
have a flat vector bundle with fibers $\Hvan{n-1}(D, \CC)$ on the complement of the
discriminant locus $P \setminus \dualX$, underlying the evident variation of Hodge
structure. Our $\Dmod$-module $\Mmod$ is in fact the minimal extension of that
flat vector bundle, and so we have a very explicit description of the minimal
extension via residues, despite the fact that the divisor $\dualX$ is typically very
singular.

The other reason is practical: The Hodge filtration on a minimal extension can be
used very nicely in the construction of certain analytic spaces, the main example
being the construction of complex-analytic N\'eron models for families of
intermediate Jacobians \cite{neron}. Good properties of the sheaves in the filtration
(such as positivity or reflexivity) translate directly into good properties of the
resulting analytic spaces (such as holomorphic convexity or control over
singularities). In a forthcoming paper, we will use the results about the sheaves $F_k
\Mmod$ obtained here to give a precise description of the N\'eron model for the
family of intermediate Jacobians $J^{n-1}(D)$, for $n = \dim X$ even. Other
applications are to the study of Noether-Lefschetz loci in $P$ (for $n$ odd).

\subsection{Summary of results}
\label{subsec:summary}

Before outlining the results of this paper, we shall briefly review the theorem of
Griffiths and Green. Keeping the notation of \subsecref{subsec:intro}, the Lefschetz
theorems show that the cohomology
groups of $D$ can be obtained from those of $X$ with the exception of $H^{n-1}(D, \QQ)$;
the latter is the direct sum of $H^{n-1}(X, \QQ)$ and the so-called \define{vanishing
cohomology} 
\[
	\Hvan{n-1}(D, \QQ) = \ker \bigl( \il \colon H^{n-1}(D, \QQ) \to H^{n+1}(X, \QQ) \bigr),
\]
the kernel of the Gysin map for the inclusion $i \colon D \into X$. The vanishing
cohomology of $D$ is related to the cohomology of the open complement $X \setminus D$
through the exact sequence
\begin{equation} \label{eq:X-D}
\begin{diagram}
	0 &\rTo& H_0^n(X, \QQ) &\rTo& H^n(X \setminus D, \QQ) 
		&\rTo& \Hvan{n-1}(D, \QQ) &\rTo& 0
\end{diagram}
\end{equation}
in which $H_0^n(X, \QQ) = \ker \bigl( \iu \circ \il \colon H^n(X, \QQ) \to H^{n+2}(X, \QQ)
\bigr)$ denotes the \define{primitive cohomology} of $X$.  The map $H^n(X \setminus
D, \QQ) \to \Hvan{n-1}(D, \QQ)$ is the well-known \define{residue map}, whose
analytic description is as follows: by a theorem of Grothendieck's, the cohomology
groups of $X \setminus D$ are computed by holomorphic forms on $X \setminus D$ with
at worst poles along the divisor $D$, and the residue map takes such a form to its
residue along $D$.

All three groups in \eqref{eq:X-D} carry mixed Hodge structures, and through
the work of Griffiths \cite{RatInt} (for the case $X = \PPn{n}$) and Green
\cite{GreenPeriod}, it is known how to compute their Hodge filtrations:
they are essentially given by pole order. That is to say, if we let $\OmX{p}(\ast D)$
denote the sheaf of meromorphic $p$-forms on $X$ with poles along $D$, and
$\OmX{p}(k D)$ the subsheaf of those with a pole of order at most $k$, then
\[
	H^n(X \setminus D, \CC) \simeq \frac{H^0 \bigl( X, \OmX{n}(\ast D) \bigr)}%
		{d H^0 \bigl( X, \OmX{n-1}(\ast D) \bigr)},
\]
and the Hodge filtration is (for $k \geq 1$) given by
\[
	F^{n+1-k} H^n(X \setminus D, \CC) \simeq \frac{H^0 \bigl( X, \OmX{n}(k D) \bigr)}%
		{d H^0 \bigl( X, \OmX{n-1}( (k-1)D ) \bigr)}.
\]
In both formulas, $d$ stands for the exterior derivative on forms. The second one
holds as long as the line bundle $\OX(D)$ is sufficiently ample; more precisely,
one needs that $H^q \bigl( X, \OmX{p}(k) \bigr) = 0$ for $p \geq 0$ and $q, k > 0$.
From \eqref{eq:X-D}, one obtains the following formula for the Hodge filtration on the
vanishing cohomology of $D$:
\begin{equation} \label{eq:Hodge-D}
	F^{n-k} \Hvan{n-1}(D, \CC) \simeq \frac{H^0 \bigl( X, \OmX{n}(k D) \bigr)}%
		{F^{n+1-k} H_0^n(X, \CC) + d H^0 \bigl( X, \OmX{n-1}( (k-1)D ) \bigr)}.
\end{equation}

The isomorphism in \eqref{eq:Hodge-D} is compatible with moving $D$ in the linear
system $P$, and provides a very convenient description for the resulting variation of
Hodge structure. Consider the incidence variety $\famX \subseteq P \times X$, which
is a projective bundle over $X$ and hence nonsingular. From now on, we denote the
individual hypersurfaces by $\famX_p = \pi^{-1}(p)$, where $\pi \colon \famX \to
P$ is the projection to the first factor. Let $\pism \colon \famXsm \to \Psm$ be the
restriction of $\pi$ to the open subset $\Psm$ over which $\pi$ is smooth, and let
$j \colon \Psm \into P$ denote the inclusion map.

On $\Psm$, we have the variation of Hodge structure 
\[
	\shH = \Rvan{n-1} \pisml \QQ = 
		\ker \bigl( R^{n-1} \pisml \QQ \to H^{n+1}(X, \QQ) \bigr),
\]
whose fibers are the weight $n-1$ rational Hodge structures on $\Hvan{n-1}(\famX_p,
\QQ)$. Let $\shHO$ be the underlying holomorphic vector bundle, and let $F^k \shHO$
be the holomorphic subbundles $F^k \shHO$ given by the Hodge filtration. The
Gauss-Manin connection $\nabla$ makes $\shHO$ into a flat vector bundle, and the
Hodge bundles satisfy Griffiths' transversality condition $\nabla(F^k \shHO)
\subseteq \Omega_{\Psm}^1 \tensor F^{k-1} \shHO$. Now the isomorphism
in \eqref{eq:Hodge-D} means that $F^{n-k} \shHO$ is a quotient of the vector bundle
\[
	\pisml \Omega_{\Psm \times X / \Psm}^n(k \famXsm) \simeq
		H^0 \bigl( X, \OmX{n}(k) \bigr) \tensor \OPsm(k),
\]
in a way that is compatible with differentation.

Trying to extend this description to singular hypersurfaces naturally leads to
filtered $\Dmod$-modules.  To avoid the problems caused by the singularities, we use
the fact that the incidence variety $\famX$ is nonsingular, and define coherent
subsheaves $F_k \Mmod$ of the quasi-coherent sheaf $\jl \shHO$ by the following rule:
for any open set $U \subseteq P$, a section $s \in \Gamma( U \cap \Psm, \shHO)$ shall
belong to $\Gamma(U, F_k \Mmod)$ if and only if there exists a
meromorphic $n$-form $\omega \in \Gamma \bigl( U \times X, \Omega_{P \times X}^n(k
\famX) \bigr)$ such that
\[
	s(p) = \Res_{\famX_p} \bigl( \omega \restr{X \setminus \famX_p} \bigr) 
\]
for any $p \in U \cap \Psm$. The result of Griffiths and Green shows that, once
$\OX(1)$ is sufficiently ample, the coherent sheaf $F_k \Mmod$ is a natural extension
of the vector bundle $F^{n-k} \shHO$ from $\Psm$ to $P$. Now let $\Mmod$ be the union
of the $F_k \Mmod$ inside $\jl \shHO$. It is then not hard to show that $(\Mmod, F)$
is a coherent filtered $\Dmod_P$-module, whose restriction to $\Psm$ is the flat
vector bundle $(\shHO, \nabla)$ with its Hodge filtration (see
\subsecref{subsec:Mmod}). The following theorem summarizes our main results.

\begin{theoremA} \label{thm:A}
If the vanishing cohomology of the hypersurfaces is nontrivial, meaning that $\shHO
\neq 0$, then the filtered $\Dmod$-module $(\Mmod, F)$ has the following properties:
\begin{enumerate}
\item $\Mmod$ is regular and holonomic, and is the minimal extension of the flat
vector bundle $(\shHO, \nabla)$ from $\Psm$ to $P$.
\item $(\Mmod, F_{\bullet+n})$ underlies the polarized Hodge module obtained by
intermediate extension of the variation of Hodge structure $\Rvan{n-1} \pisml \QQ$,
provided that $H^q \bigl( X, \OmX{p}(k) \bigr) = 0$ for every $p \geq 0$ and every $q, k > 0$.
\item By the Riemann-Hilbert correspondence, the de Rham complex $\DR_P(\Mmod)$ is
constructible, and its cohomology sheaves satisfy
\[
	R^{n - 1 + k} \pil \CC \simeq \cohshf^{k-d} \DR_P \bigl( \Mmod \bigr) 
		\oplus H^{n - 1 - k}(X, \CC) 
\]
for all $0 \leq k \leq d$.
\end{enumerate}
\end{theoremA}

\begin{note}
The conditions in the theorem are always satisfied if $\OX(1)$ is a sufficiently high
power of a very ample line bundle.
\end{note}

In the process of proving Theorem~\ref{thm:A}, we compute the characteristic variety
of the $\Dmod$-module $\Mmod$ inside the cotangent bundle of $P$. Before we state the
result, recall that $\Gr^F \Dmod_P \simeq \Sym \Theta_P$ is the symmetric algebra on
the tangent bundle of $P$; consequently, the graded module $\Gr^F \Mmod$ defines a
coherent sheaf on $T_P^{\ast}$, whose support is by definition the characteristic
variety. Because of the grading, we also get a coherent sheaf on the projectivized
cotangent bundle $\PP(\Theta_P)$. One can show (see \subsecref{subsec:universal})
that the incidence variety $\famX$ embeds into $\PP(\Theta_P)$. Denote
by $\famY \subseteq \famX$ the union of all the singular points in the hypersurfaces
$\famX_p$; then the second projection $\psi \colon \famY \to X$ is again a projective
bundle, and so $\famY$ is also a smooth subvariety of $\PP(\Theta_P)$. The following
theorem relates it to the characteristic variety of $(\Mmod, F)$.

\begin{theoremA} \label{thm:B}
If the vanishing cohomology of the hypersurfaces is nontrivial, then
the coherent sheaf on $\PP(\Theta_P)$ defined by the graded $\Sym \Theta_P$-module
$\Gr^F \Mmod$ is precisely $\psiu \omega_X$.
\end{theoremA}

We note that this result fits into the context of Fourier-Mukai transforms for filtered
$\Dmod$-modules \cite{Laumon}, in this case between $X$ and on $P$. That is to say,
if we transform $\OX$ (with the trivial filtration $\Gr_0^F \OX = \OX$) by the kernel
$\OPX(\ast \famX)$, then $(\Mmod, F)$ is a direct summand in one of the cohomology
modules of the resulting complex. More generally, because of the embedding $i \colon
X \into Q$, this may be seen as a special case of the Fourier-Mukai transform between filtered
$\Dmod$-modules on the projective space $Q$ and its dual $P$. (It appears that this
observation is originally due to Beilinson.)

This computation also leads to the following result about the terms that appear in
the Decomposition Theorem for the direct image of $\QQ \decal{\dfamX}$ under the morphism
$\pi \colon \famX \to P$.

\begin{theoremA} \label{thm:C}
If the vanishing cohomology of the hypersurfaces is nontrivial, then each 
piece in the decomposition of the polarized Hodge modules $H^k \pil \QfamXH
\decal{\dfamX}$ has strict support equal to all of $P$.
\end{theoremA}

As explained above, we view the coherent sheaves $F_k \Mmod$ (especially in the range $1 \leq
k \leq n$) as being natural extensions of the Hodge bundles $F^{n-k} \shHO$. The
concrete description by residues allows us to show that they have many nice
properties.

\begin{theoremA} \label{thm:D}
Suppose that we have $H^q \bigl( X, \OmX{p}(k) \bigr) = 0$ for $k,q > 0$ and $p \geq
0$. Then the coherent sheaves $F_k \Mmod$ have the following properties:
\begin{enumerate}[label=(\arabic{*}), ref=(\arabic{*})]
\item $F_1 \Mmod$ is an ample vector bundle, and is a direct summand of $\pil
\OfamX(K_{\famX/P})$.
\item $H^i \bigl( P, \Omega_P^p \tensor F_k \Mmod \bigr) = 0$ for every $i \geq
\max(p,0)$.
\item For $0 \leq i \leq \dim P - 1$, we have
\[
	\Ext_p^i \bigl( F_k \Mmod, \OP \bigr) \simeq
		\left( \frac{F^{n+1-k} H^{n+1-i}(X, \CC)}{F^{n-k} H^{n-1-i}(X, \CC)} \right)^{\vee}.
\]
\item Given $m \geq 1$, the sheaves $\Gr_k^F \Mmod$ and $F_k \Mmod$ in the range $1
\leq k \leq n$ satisfy a variant of Serre's condition $S_m$, provided that $\OX(1)$ is a
sufficiently high power of a very ample line bundle. In particular, they are
reflexive. \label{en:thmC-4}
\end{enumerate}
\end{theoremA}

The proof of \ref{en:thmC-4}, given in \subsecref{subsec:Serre}, depends on two
things: a general duality result for filtered Cohen-Macaulay $\Dmod$-modules
\cite{mhmduality}, and the fact that the set of hypersurfaces with ``many''
singularities has very high codimension in $P$ if $\OX(1)$ is a sufficiently large
power of a very ample line bundle.

\subsection{Acknowledgements}

This work is part of my Ph.D.~thesis, and it is my pleasure to thank my former
adviser Herb Clemens for his help and his kindness during my time in graduate school.
Needless to say, the idea to study the behavior of the residue map near singular
hypersurfaces is due to him. 

I also acknowledge the influence of Patrick Brosnan and Greg Pearlstein, whose paper
\cite{BFNP} with Hao Fang and Zhaohu Nie appeared at about the time when I was
thinking about residues, and led me to describe the $\Dmod$-module $(\Mmod, F)$ using
the theory of mixed Hodge modules.

I am grateful to Joel Kamnitzer and Sabin Cautis for useful conversations about
Fourier-Mukai transforms for filtered $\Dmod$-modules, and for calling my attention
to Laumon's work \cite{Laumon}.

Lastly, my thanks go to Paul Taylor for providing the package \texttt{diagrams} that
was used to typeset the many commutative diagrams in the paper.

\subsection{Conventions}
\label{subsec:conventions}

In dealing with filtrations, we index increasing filtrations (such as weight
filtrations, or Hodge filtrations on $\Dmod$-modules) by lower indices, and decreasing
filtrations (such as Hodge filtrations on vector spaces, or $V$-filtrations on left
$\Dmod$-modules) by upper indices. We may pass from one to the other by the
convention that $F^{\bullet} = F_{-\bullet}$. To be consistent, shifts in the
filtration thus have different effects in the two cases:
\[
	F \decal{1}^{\bullet} = F^{\bullet+1}, \quad \text{while} \quad
	F \decal{1}_{\bullet} = F_{\bullet-1}.
\]
This convention agrees with the notation used in M.~Saito's papers. 

In this paper, we work with \emph{left} $\Dmod$-modules, and $\Dmod$-module always
means left $\Dmod$-module (in contrast to \cite{SaitoMHM}, where right
$\Dmod$-modules are used).

When $M$ is a mixed Hodge module, the effect of a Tate twist $M(k)$ on the underlying
filtered $\Dmod$-module $(\Mmod, F)$ is as follows: 
\[
	(\Mmod, F)(k) = \bigl( \Mmod, F \decal{k} \bigr) = 
		\bigl( \Mmod, F_{\bullet-k} \bigr).
\]

For a regular holonomic $\Dmod$-module $\Mmod$ that is defined on the complement of a
divisor $D \subseteq X$, we let $\jlreg \Mmod$ be the direct image in the category of
regular holonomic $\Dmod$-modules; its sections have poles of finite order along $D$.

The individual sections of the paper are numbered consecutively, and are referred to
with a paragraph symbol (such as \subsecref{subsec:conventions}).

\section{Residues and filtered D-modules}

Throughout, we let $X$ be a smooth projective variety of dimension $n$, and 
$\OX(1)$ a very ample line bundle, embedding $X$ into the projective space $Q =
\PP(V)$, where $V = H^0 \bigl( X, \OX(1) \bigr)$. We denote by $P$ the dual projective
space, parametrizing hyperplane sections of $X$ for the given embedding, and let
$\pi \colon \famX \to P$ be the universal hypersurface. 

\subsection{Filtered $\Dmod$-modules on projective space}
\label{subsec:dmodule}

Let $V$ be a complex vector space of dimension $d+1$, and set $P = \PP(\dualV)$,
which is a projective space of dimension $d$.
Let $\Dmod_P$ denote the sheaf of differential operators on $P$; it is naturally
filtered, with the subsheaf $F_m \Dmod_P$ consisting of those differential operators
whose order is at most $m$. The associated graded sheaf
\[
	\Gr^F \Dmod_P = \bigoplus_{m=0}^{\infty} F_m \Dmod_P \big/ F_{m-1} \Dmod_P
\]
is commutative, and in fact isomorphic to the symmetric algebra on the tangent bundle
$\Sym \Theta_P$. Recall that a \define{left $\Dmod$-module} on $P$ is a quasi-coherent sheaf $\Mmod$
with a left action by $\Dmod_P$. Moreover, a \define{filtered $\Dmod$-module} is a pair
$(\Mmod, F)$, consisting of left a $\Dmod$-module $\Mmod$ together with an increasing
filtration $F = F_{\bullet} \Mmod$ by $\OP$-coherent subsheaves that is bounded from
below, and satisfies
\begin{equation} \label{eq:filtration}
	F_m \Dmod_P \cdot F_k \Mmod \subseteq F_{m+k} \Mmod
\end{equation}
for all $k,m \in \ZZ$. One says that the filtration $F$ is \define{good} if the
inclusion in \eqref{eq:filtration} is an equality for large values of $k$; this is
equivalent to
\[
	\Gr^F \! \Mmod = \bigoplus_{k \in \ZZ} F_k \Mmod \big/ F_{k-1} \Mmod
\]
being finitely generated over $\Sym \Theta_P$ by \cite{Borel}*{Theorem~4.4}.
Note that $\Mmod$ is coherent as an $\OP$-module if, and only if, it is locally free
of finite rank (and hence a flat vector bundle); see \cite{Borel}*{Proposition~1.7}.

$\Sym \Theta_P$ is naturally the sheaf of functions on the cotangent bundle of
$P$, and so any filtered coherent $\Dmod$-module $(\Mmod, F)$ defines a
coherent sheaf on $T_P^{\ast}$; its support is by definition the \emph{characteristic
variety} $\Ch(\Mmod, F)$. Again, the characteristic variety is contained in the zero
section exactly when $\Mmod$ is locally free of finite rank.

Because of the grading, $\Ch(\Mmod, F)$ is a cone in $T_P^{\ast}$, and 
each component of this cone has dimension at least $d = \dim P$, a fact known as
Bernstein's inequality \cite{Borel}*{Theorem~1.10}.  The filtered $\Dmod$-module
$(\Mmod, F)$ is said to be \define{holonomic} if its characteristic variety is of
pure dimension $d$. For every holonomic $\Dmod$-module $\Mmod$, there is a finite
union of closed subvarieties of $P$, outside of which $\Mmod$ is a flat vector
bundle, namely the image of all those components of $\Ch(\Mmod, F)$
that are not contained in the zero section.  For a holonomic $\Dmod$-module $\Mmod$,
the \define{multiplicity} of the sheaf $\Gr^F \! \Mmod$ along each component of the
characteristic variety is the same for every good filtration $F$. These
multiplicities are additive in short exact sequences of holonomic $\Dmod$-modules,
which implies that the category of holonomic $\Dmod$-modules is Artinian: every
holonomic module admits a finite filtration with simple holonomic subquotients.
This leads to the following notion: given a holonomic $\Dmod$-module $\Nmod$ on $P
\setminus Z$, we say that a holonomic $\Dmod$-module $\Mmod$ is the \define{minimal
extension} of $\Nmod$ if $\Mmod \restr{P \setminus Z} \simeq \Nmod$, and if $\Mmod$
has no nontrivial submodules or quotient modules whose support is contained in $Z$.

Given a filtered coherent $\Dmod$-module $(\Mmod, F)$, the graded $\Sym \Theta_P$-module
$\Gr^F \! \Mmod$ naturally defines a coherent sheaf on the projectivized cotangent
bundle $\PP(\Theta_P) = \Proj(\Sym \Theta_P)$; we shall call it the
\define{characteristic sheaf} of $(\Mmod, F)$, and denote it by $\shC(\Mmod, F)$.
We recover the characteristic variety by taking the cone over the support of
the characteristic sheaf, together with possibly the zero section; in particular, one can
compute the multiplicity along the components of the characteristic variety from
$\shC(\Mmod, F)$.

A second interpretation will be useful in what follows. Let $Q = \PP(V)$, and note
that the two projective spaces $P$ and $Q$ are dual to each other, with points $p \in
P$ corresponding to hyperplanes $H_p \subseteq Q$. The incidence variety
\[
	\famI = \menge{(p,x) \in P \times Q}{x \in H_p}
\]
is a nonsingular hypersurface, with line bundle $\OPQ(\famI) \simeq
\prPu \OP(1) \tensor \prQu \OQ(1)$. It is naturally isomorphic to the projectivized
cotangent bundle $\PP \bigl( \Theta_P \bigr)$; indeed, the exactness of the
Euler sequence
\begin{diagram}
	0 &\rTo& \OP &\rTo& \OP(1) \tensor V &\rTo& \Theta_P &\rTo& 0
\end{diagram}
on the projective space $P$ implies that $\PP \bigl( \Theta_P  \bigr)$ embeds into
$\PP \bigl( \OP(1) \tensor V \bigr) = P
\times Q$, with line bundle $\prPu \OP(1) \tensor \prQu \OQ(1)$, and a moment's thought
shows that the image is precisely $\famI$.

Let $i \colon \famI \into P \times Q$ be the inclusion. The line bundle $\OPQ(\famI)$
is very ample, and embeds $P \times Q$ into a larger projective space; we let
\[
	S = \bigoplus_{k=0}^{\infty} H^0 \bigl( \famI, \iu \OPQ(k \famI) \bigr)
\]
be the homogeneous coordinate ring in this embedding, so that $\PP(\Theta_P) \simeq \Proj
S$. Not surprisingly, we have an isomorphism of graded rings
\[
	S \simeq H^0 \bigl( P, \Sym \Theta_P \bigr);
\]
in fact, after pushing forward the exact sequence 
\begin{diagram}
	0 &\rTo& \OPQ \bigl( (k-1)\famI \bigr) &\rTo& \OPQ(k\famI) &\rTo& \iu \OPQ(k\famI)
		&\rTo& 0
\end{diagram}
to $P$, and using the information coming from the Euler sequence, we find that
\[
	H^0 \bigl( \famI, \iu \OPQ(k\famI) \bigr) \simeq 
		H^0 \left( P, \frac{\OP(k) \tensor \Sym^k V}{\OP(k-1) \tensor \Sym^{k-1} V} \right)
		\simeq H^0 \bigl( P, \Sym^k \Theta_P \bigr)
\]
for all $k \geq 0$. It follows from these observations that the characteristic sheaf
$\Ch(\Mmod, F)$ of a filtered coherent $\Dmod$-module $(\Mmod, F)$ is the coherent
sheaf on $\Proj S$ associated to the graded $S$-module
\begin{equation} \label{eq:charmod}
	C(\Mmod, F) = \bigoplus_{k \in \ZZ} H^0 \bigl( P, F_k \Mmod / F_{k-1} \Mmod \bigr)
		= H^0 \bigl( P, \Gr^F \! \Mmod \bigr).
\end{equation}
We shall refer to it as the \define{characteristic module} of $(\Mmod, F)$; as usual,
the characteristic sheaf is completely determined by the homogeneous components of
$C(\Mmod, F)$ in degrees $k \gg 0$.

\subsection{The universal hyperplane section}
\label{subsec:universal}

Now let $X$ be a smooth projective variety of dimension $n$, and fix a very ample
line bundle $\OX(1)$ on $X$. We let $V = H^0(X, \OX(1))$ be the space of
its global sections, and therefore get a nondegenerate embedding $i \colon X \into Q$ into the projective
space $Q = \PP(V)$. The \define{universal hyperplane section} is the incidence variety
\[
	\famX = \menge{(p,x) \in P \times X}{x \in H_p \cap X}.
\]
We obviously have $\famX = \famI \cap (P \times X)$ inside $P \times Q$, in the
notation of \subsecref{subsec:dmodule}. Because $\famI \simeq \PP(\Theta_Q)$, it
follows that 
\[
	\famX \simeq \PP(\iu \Theta_Q)
\]
is itself a projective bundle of rank $d-1$ over $X$; consequently, $\famX$ is a
nonsingular very ample hypersurface in $P \times X$ of dimension $n+d-1$, with
corresponding line bundle $\OPX(\famX) \simeq \prPu \OP(1) \tensor \prXu \OX(1)$. Let
$\pi \colon \famX \to P$ be the projection to the first coordinate; its fibers
$\famX_p = \pi^{-1}(p)$ are the hyperplane sections of $X$.

We shall also use the subvariety $\famY \subseteq \famX$, defined by the condition
\[
	\famY = \menge{(p,x) \in P \times X}{\text{$H_p$ is tangent to $X$ at $x$}};
\]
its points are exactly the singular points in the fibers of $\pi$. Note that $\famY$
is itself nonsingular of dimension $d-1$, since it is isomorphic to the projective
bundle $\PP(\NXQ)$ over $X$; this follows from the exact sequence
\begin{equation} \label{eq:normal}
\begin{diagram}
0 &\rTo& \Theta_X &\rTo& \iu \Theta_Q &\rTo& \NXQ &\rTo& 0.
\end{diagram}
\end{equation}
The dual variety $\dualX \subseteq P$ is by definition the image $\pi(\famY)$. If
there are points in $P$ corresponding to hyperplane sections with just a single ordinary
double point singularity, then $\famY$ is birational to $\dualX$ and hence a
resolution of singularities.

\begin{note}
Both $\famX$ and $\famY$ can be considered as subvarieties of the projectivized cotangent
bundle $\PP(\Theta_P) \simeq \famI$. In what follows, we shall see that both
varieties are the support of the characteristic sheaves of two natural filtered $\Dmod$-modules on $P$.
\end{note}

\subsection{An auxilliary $\Dmod$-module}

Having completed our discussion of the universal hyperplane section, we now turn to
the problem of extending the residue calculus to singular hyperplane sections of $X$.
As above, let $\pi \colon \famX \to P$ be the universal hyperplane section.
To begin with, we define
\[
	\Nmod = \prPl \OmPXP{n}(\ast \famX),
\]
which is a quasi-coherent sheaf on $P$ whose sections over an open set $U$ are
relative rational $n$-forms on $U \times X$ with poles along $\famX$.
It has a natural increasing filtration by coherent subsheaves $F_k
\Nmod$, corresponding to the order of the pole; set
\[
	F_k \Nmod = \prPl \OmPXP{n} (k \famX) \simeq
		H^0 \bigl( X, \OmX{n}(k) \bigr) \tensor \OP(k),
\]
as well as $F_k \Nmod = 0$ for $k \leq 0$ to avoid degenerate cases.

We observe that $\Nmod$ is naturally a left $\Dmod$-module, since the sheaf of
differential operators $\DP$ acts on $\Nmod$ by differentiating the
coefficients. More precisely, for $\omega \in \Gamma(U, \Nmod)$ a relative rational
$n$-form on $U \times X$, and $\xi \in \Gamma(U, \Theta_P)$ a vector field, we define 
\[
	\xi \cdot \omega = \contr{\xi'}{d \omega} \in \Gamma(U, \Nmod),
\]
where $\xi'$ is the natural horizontal lifting of $\xi$ to a vector field on $U
\times X$, and $\contr{\xi'}{d \omega}$ indicates contraction. It is not hard to prove the following: (1) for a holomorphic function $f$
on $U$, we have $(f \xi) \cdot \omega = f (\xi \cdot \omega)$, as well as $\xi \cdot
(f \omega) = (\xi f) \cdot \omega + f (\xi \cdot \omega)$; (2) for a second vector
field $\eta$, we have $\lie{\xi}{\eta} \cdot \omega = \xi \cdot (\eta \cdot \omega) -
\eta \cdot (\xi \cdot \omega)$; (3) if $\omega$ is a section of $F_k \Nmod$, then
$\xi \cdot \omega$ is a section of $F_{k+1} \Nmod$. The action by vector fields thus
extends uniquely to an action by $\DP$, and so $\Nmod$ is a left $\Dmod$-module;
moreover, we have $F_m \Dmod_P \cdot F_k \Nmod \subseteq F_{m+k} \Nmod$, where
$F_{\bullet} \Dmod_P$ is the order filtration on $\Dmod_P$.

\begin{lemma} \label{lem:filtered}
$F_{\bullet} \Nmod$ is a good filtration, and $(\Nmod, F)$ is a coherent
filtered $\Dmod$-module.
\end{lemma}

\begin{proof}
To prove that the filtration is good, we need to show that $F_1 \Dmod_P \cdot
F_k \Nmod = F_{k+1} \Nmod$ for all sufficiently large $k$. This is equivalent to the
surjectivity of the map
\begin{equation} \label{eq:surj}
	\Theta_P \tensor F_k \Nmod \to F_{k+1} \Nmod / F_k \Nmod,
\end{equation}
a question which is local on $P$. Fix an arbitrary point $p \in P$, and let $s_0, s_1,
\dotsc, s_d$ be a basis for the vector space $H^0(X, \OX(1))$, with
$\lbrack s_0 \rbrack$ equal to the point $p$. Any local section of $F_k \Nmod$ can
now be written in the form
\[
	\omega = \frac{\omega(t)}{\bigl( s_0 + \sum t_i s_i \bigr)^k},
\]
for some holomorphic map $\omega(t)$ from $\CC^d$ into $H^0 \bigl( X, \OmX{n}(k)
\bigr)$. Setting $\partial_i = \partial/\partial t_i$, we then have
\[
	\bigl( \partial_i \cdot \omega \bigr) \big\vert_{t=0} = 
		- k \cdot \frac{s_i \omega(0)}{s_0^{k+1}} 
			+ \frac{\partial_i \omega(t) \restr{t=0}}{s_0^k}
		\equiv -k \cdot \frac{s_i \omega(0)}{s_0^{k+1}} \mod F_k \Nmod.
\]
Thus surjectivity of \eqref{eq:surj} on stalks is equivalent to surjectivity
of the product maps
\[
	H^0 \bigl( X, \OX(1) \bigr) \tensor H^0 \bigl( X, \OmX{n}(k) \bigr)
		\to H^0 \bigl( X, \OmX{n}(k+1) \bigr),
\]
which is satisfied for $k \gg 0$ because $\OX(1)$ is very ample.
\end{proof}

Since it is illustrative, we shall also compute the characteristic variety of
$\Nmod$; more precisely, we shall compute the characteristic sheaf $\shC(\Nmod, F)$.
To state the result, let $\phi \colon \famX \to X$ be the projection to the second
coordinate, and recall that $\famX$ is naturally embedded into $\PP(\Theta_P)$.

\begin{lemma} \label{lem:charN}
We have $\shC(\Nmod, F) \simeq \phiu \OmX{n}$; consequently, the characteristic
variety of $\Nmod$ is the cone over $\famX$, and thus has a single irreducible component
of dimension $\dim \famX + 1 = d + n$ and multiplicity one.
\end{lemma}

\begin{proof}
As explained in \subsecref{subsec:dmodule}, the coherent sheaf $\shC(\Nmod, F)$ is
determined by its module of sections
\[
	C(\Nmod, F) = \bigoplus_{k \in \ZZ} 
		H^0 \bigl( P, F_k \Nmod / F_{k-1} \Nmod \bigr).
\]
To compute the individual summands, recall that
\[
	F_k \Nmod = \prPl \OmPXP{n}\bigl( k \famX \bigr) \simeq 
		H^0 \bigl( X, \OmX{n}(k) \bigr) \tensor \OP(k)
\]
for $k \geq 1$ (and $0$ otherwise). Thus $H^1(P, F_{k-1} \Nmod) = 0$, and so
\[
	H^0 \bigl( P, F_k \Nmod \big/ F_{k-1} \Nmod \bigr) \simeq 
		\frac{H^0 \bigl( P, F_k \Nmod \bigr)}{H^0 \bigl( P, F_{k-1} \Nmod \bigr)} \simeq 
		\frac{H^0 \bigl( P \times X, \OmPXP{n}(k \famX) \bigr)}
			{H^0 \bigl( P \times X, \OmPXP{n}((k-1) \famX) \bigr)}.
\]
Writing $\OfamX(1)$ for the restriction of $\OPX(\famX)$ to the universal hyperplane
section, we have the exact sequence
\begin{diagram}
	0 &\rTo& \OmPXP{n} \bigl( (k-1) \famX \bigr) &\rTo& 
		\OmPXP{n} \bigl( k \famX \bigr) &\rTo&
		\phiu \OmX{n} \tensor \OfamX(k) &\rTo& 0.
\end{diagram}
The higher cohomology of the first term being zero for $k \gg 0$, we thus get
\[
	H^0 \bigl( P, F_k \Nmod \big/ F_{k-1} \Nmod \bigr) \simeq 
		H^0 \bigl( \famX, \phiu \OmX{n} \tensor \OfamX(k) \bigr),
\]	
and hence $C(\Nmod, F)$ agrees in large degrees with the graded module
\[
	\bigoplus_{k = 0}^{\infty} H^0 \bigl( \famX, \phiu \OmX{n} \tensor \OfamX(k) \bigr).
\]
This means that the characteristic sheaf $\shC(\Nmod, F)$ is precisely $\phiu
\OmX{n}$, viewed as a sheaf on $\PP(\Theta_P)$ under the inclusion $\famX \into
\PP(\Theta_P)$.
\end{proof}

\subsection{The object of interest}
\label{subsec:Mmod}

We now define the $\Dmod$-module that we are actually interested in. Let $\shHO$ be the
vector bundle
\[
	\shHO = \OPsm \tensor_{\CC} \Rvan{n-1} \pisml \CC
\]
with fibers $\Hvan{n-1}(\famX_p, \CC)$. It underlies the variation of Hodge structure
$\Rvan{n-1} \pisml \QQ$ on the vanishing cohomology of the hyperplane sections. Let
$j \colon \Psm \into P$ be the inclusion. The operation of taking fiberwise residues
defines a map of sheaves
\[
	\ResfamXP \colon \Nmod \to \jl \shHO;
\]
concretely, for $\omega \in \Gamma(U, \Nmod)$, we let $\ResfamXP(\omega) \in
\Gamma(U \cap \Psm, \shHO)$ be the section whose value at a point $p \in U \cap \Psm$
is 
\[
	\Res_{\famX_p} \bigl( \omega \restr{X \setminus \famX_p} \bigr) \in 
		\Hvan{n-1}(\famX_p, \CC).
\]
We define $\Mmod$ as the image sheaf, and also let $F_k \Mmod$ be the image of
$F_k \Nmod$. Over $\Psm$, taking residue commutes with the action by vector fields
\cite{Voisin}*{p.~425--6}, and we can therefore give $\Mmod$ the
induced $\Dmod$-module structure. Then $F_{\bullet} \Mmod$ is a good filtration on
$\Mmod$, and $\ResfamXP \colon \Nmod \to \Mmod$ is a map of filtered $\Dmod$-modules. 

Clearly, $\ResfamXP$ is surjective over $\Psm$, and so we have $\ju \Mmod = \shHO$;
this is consistent with the $\Dmod$-module structure on $\shHO$ given by the
Gauss-Manin connection. Moreover, once $\OX(1)$ is sufficiently ample, we even have
\[
	\ju F_k \Mmod = F^{n-k} \shHO.
\]
because the Hodge filtration is determined by pole order. Thus $(\Mmod, F)$ is an
extension of $(\shHO, F)$ to a filtered $\Dmod$-module on $P$; we will see later that
$\Mmod$ is, in fact, the minimal extension in the category of holonomic
$\Dmod$-modules.

\section{Mixed Hodge modules}

In this section, we investigate the relationship between our filtered $\Dmod$-module
and M.~Saito's theory of mixed Hodge modules; the main point is that, up to a shift
in the filtration, $(\Mmod, F)$ underlies a Hodge module. 

\subsection{Mixed Hodge modules}

Instead of reviewing Saito's theory in the abstract, we shall concentrate on a
specific example: how to prove the following familiar result using mixed Hodge
modules.

\begin{proposition} \label{prop:HFrat}
Let $X$ be a smooth complex projective variety of dimension $n$, and let $D \subseteq
X$ be a smooth hypersurface. Assume that $\OX(D)$ is ample enough so that
$H^q \bigl( X, \OmX{p}(kD) \bigr) = 0$ for $q > 0$ and $k > 0$.
Then the cohomology in degree $k$ of the complex
\begin{diagram}
	H^0 \bigl( X, \OmX{p}(D) \bigr) &\rTo^d&
		H^0 \bigl( X, \OmX{p+1}(2D) \bigr) &\rToDots&
		H^0 \bigl( X, \OmX{n} \bigl( (n-p+1)D \bigr) \bigr)
\end{diagram}
is isomorphic to $F^p H^k(X \setminus D, \CC)$. The isomorphism takes a $d$-closed
rational form to the cohomology class defined by its restriction to $X \setminus D$.
\end{proposition}

For $X$ any quasi-projective complex algebraic variety, $\MHM(X)$ denotes the abelian
category of algebraic mixed Hodge modules on $X$. Each mixed Hodge module $M$ has an
underlying perverse sheaf $\rat M$, and the functor
\[
	\rat \colon \MHM(X) \to \Pervalg{X}
\]
to the category of $\QQ$-valued algebraic perverse sheaves on $X$ is fully faithful.
(Here ``algebraic'' means that we only consider stratifications whose strata are
algebraic varieties.) Each mixed Hodge module also has an underlying filtered
holonomic $\Dmod$-module $(\Mmod, F)$, which corresponds to (the complexification of)
$\rat M$ under the Riemann-Hilbert correspondence; the isomorphism $\rat M
\tensor_{\QQ} \CC \simeq \DR_X(\Mmod)$
should be seen as providing the $\Dmod$-module $\Mmod$ with a $\QQ$-structure.
In order for $M$ to be a mixed Hodge module, several very restrictive conditions have
to be satisfied (for instance, existence of a weight filtration, admissibility,
good behavior under nearby and vanishing cycle functors, to name just a few); since
we do not need them here, we refer to \cite{SaitoHM}*{\S4} for details.

The basic examples of mixed Hodge modules are (admissible, graded-polarized)
variations of mixed Hodge structure. In the case of a point, $\MHM(\pt)$ is the
category of (graded-polarized) mixed Hodge structures defined over $\QQ$. Following
Saito, we let $\QQ^H \in \MHM(\pt)$ be the unique Hodge structure on $\QQ$ of weight
zero. With $d_X = \dim X$ and $a_X \colon X \to \pt$ the structure map, we also define
$\QXH = \aXu \QQ^H \in D^b \MHM(X)$. When $X$ is smooth, $\QXH \decal{d_X}$
is a pure Hodge module on $X$ whose underlying perverse sheaf is $\QQ_X \decal{d_X}$;
the corresponding $\Dmod$-module is $\OX$, with filtration
\[
	F_k \OX = \begin{cases} 
		\OX & \text{if $k \geq 0$,} \\
		0   & \text{otherwise.}
	\end{cases}
\]

Now suppose that $X$ is a smooth projective variety of dimension $n$, and that $D
\subseteq X$ is a smooth divisor. The cohomology of $U = X \setminus D$ is governed
by the mixed Hodge module $\jl \QQ_U^H \decal{n}$ on $X$, where $j \colon U
\into X$ is the inclusion map. Since $D$ has no singularities, the underlying
filtered $\Dmod$-module is easy to describe: it is the sheaf $\Mmod_U = \OX(\ast D)$ of
rational functions with poles along $D$, and the filtration
\[
	F_k \Mmod_U = \begin{cases}
		\OX \bigl( (k+1)D \bigr) & \text{if $k \geq 0$,} \\
		0 & \text{otherwise,}
	\end{cases}
\]
is given by pole order (see \cite{SaitoB}*{Corollary~4.3} for more details).

To compute the Hodge filtration on the cohomology groups $H^i(U, \QQ)$, we observe
that, as mixed Hodge structures,
\[
	H^i(U) \simeq H^{i-n} a_{U \ast} \QQ_U^H \decal{n}
		\simeq H^{i-n} \aXl \bigl( \jl \QQ_U^H \decal{n} \bigr).
\]
The underlying vector space can be computed with the help of the de Rham complex of
the $\Dmod$-module $\Mmod_U$; this is the complex
\begin{diagram}
	\DR_X(\Mmod_U) = \Bigl\lbrack
		\Mmod_U &\rTo^d& \OmX{1} \tensor \Mmod_U &\rToDots& 
		\OmX{n} \tensor \Mmod_U \Bigr\rbrack \decal{n},
\end{diagram}
supported in degrees $-n, \dotsc, 0$. Its hypercohomology group $\HH^k \bigl(
\DR_X(\Mmod_U) \bigr)$ is the $\CC$-vector space underlying the mixed Hodge
structure $H^k a_{U \ast} \QQ_U^H \decal{n}$. 

The de Rham complex is naturally filtered by the subcomplexes
\begin{diagram}
	F_k \DR_X(\Mmod_U) = \Bigl\lbrack
		F_k \Mmod_U &\rTo& \OmX{1} \tensor F_{k+1} \Mmod_U &\rToDots& \OmX{n} 
		\tensor F_{k+n} \Mmod_U \Bigr\rbrack \decal{n},
\end{diagram}
and the induced filtration on hypercohomology, and hence on $H^k a_{U \ast} \QQ_U^H
\decal{n}$, is exactly the Hodge filtration. Moreover, since $X$ is smooth and projective, the
filtration is actually strict, meaning that the map
$\HH^k \bigl( F_k \DR_X(\Mmod_U) \bigr) \to \HH^k \bigl( \DR_X(\Mmod_U) \bigr)$
is injective. It is now very easy to compute the Hodge filtration on $H^i(U, \CC)$.
Indeed, 
\[
	F^p H^i(U, \CC) \simeq F_{-p} \, \HH^{i-n} \bigl( \DR_X(\Mmod_U) \bigr)
		\simeq \HH^{i-n} \bigl( F_{-p} \DR_X(\Mmod_U) \bigr),
\]
and the term in degree $i-n$ of the indicated complex is
\[
	\OmX{n+(i-n)} \tensor F_{-p+n+(i-n)} \Mmod_U = \begin{cases}
		\OmX{i} \bigl( (i-p+1) D \bigr) & \text{for $i \geq p$,} \\
		0 & \text{otherwise.}
	\end{cases}
\]
Under the ampleness assumptions on $D$, each of these sheaves has no higher
cohomology, and so we obtain the assertion of Proposition~\ref{prop:HFrat}.

\subsection{Cohomology of the universal hypersurface}

We take from \cite{BFNP} the idea of studying the universal hypersurface by means
of mixed Hodge modules. Where convenient for the purpose of this paper, we reprove
some of their results using the theory of $\Dmod$-modules instead of perverse
sheaves. For notational economy, let $d = \dim P$, $n = \dim X$, and $\dfamX = n + d
-1$.

The crucial point is that $\famX$ is nonsingular; this makes it possible to use the
Decomposition Theorem. Recall from \cite{BFNP}*{4.5} that there is a decomposition
\begin{equation} \label{eq:decomp}
	\pil \QfamXH \decal{\dfamX} \simeq \bigoplus_{i,j} E_{i,j} \decal{-i}
\end{equation}
in the derived category $D^b \MHM(P)$. Indeed, since $\famX$
is smooth and projective, the mixed Hodge module $\QfamXH \decal{\dfamX}$ on $\famX$
is pure of weight $\dfamX$; the same is then true for its image under $\pil$ because
$\pi$ is smooth and projective. Applying Saito's version of the Decomposition
Theorem, we have 
\[
	\pil \QfamXH \decal{\dfamX} \simeq \bigoplus_i E_i \decal{-i},
\]
where $E_i = H^i \pil \QfamXH \decal{\dfamX}$. We can further decompose each $E_i$
into simple Hodge modules; let $E_{i,j}$ be the direct sum of all those pieces whose
codimension of strict support is equal to $j$. We then arrive at the isomorphism in
\eqref{eq:decomp}, with $E_{i,j}$ pure of weight $\dfamX + i$ and supported on a
closed subscheme of $P$ of codimension $j$. Moreover, we see that
\begin{equation} \label{eq:decompk}
	H^i \pil \QfamXH \decal{\dfamX} \simeq \bigoplus_{j = 0}^d E_{i,j}
\end{equation}
for all $i \in \ZZ$. This decomposition is most interesting when $i=0$.

One general result about the decomposition is Saito's Hard Lefschetz Theorem
\cite{SaitoMHM}*{Th\'eor\`eme~1 on p.~853}, which asserts that
\begin{equation} \label{eq:SHL}
	E_{i,j} \simeq E_{-i,j}(-i);
\end{equation}
note that it again depends on the fact that $\pi$ is projective and smooth.
Essentially all the individual Hodge modules $E_{i,j}$ have been computed in
\cite{BFNP}; for the convenience of the reader, we summarize the main results:
\begin{enumerate}[label=(\alph{*}), ref=(\alph{*})]
\item Perverse Weak Lefschetz Theorem: $E_{i,j} = 0$ unless $i=0$ or $j=0$.
\item $E_{i,0} \simeq H^{n+i-1}(X) \tensor \QPH \decal{d}$ for $i < 0$.
\item $E_{0,0} \simeq \jlsl V^{n-1} \simeq \jlsl \Vvan \oplus H^{n-1}(X) \tensor \QPH
\decal{d}$.
\item $E_{0,1} = 0$ iff the vanishing cohomology $\Vvan$ is nonzero.
\item $E_{0,j} = 0$ for all $j > 0$, provided that $\OX(1)$ is sufficiently ample.
\end{enumerate}
To understand the notation, recall that $j \colon \Psm \into P$ is the inclusion of
the open set where the map $\pi \colon \famX \to P$ is smooth. Let $V^{n-1} =
H^{n-1} \pisml \QQ_{\famXsm}^H \decal{d}$ denote the Hodge module on $\Psm$ whose
underlying perverse sheaf is the shifted local system $R^{n-1} \pisml \QQ \decal{d}$;
it admits a direct sum decomposition
\begin{equation} \label{eq:V-direct-sum}
	V^{n-1} = \Vvan \oplus H^{n-1}(X) \tensor \QQ_{\Psm}^H \decal{d},
\end{equation}
where $\Vvan$ is the polarized variation of Hodge structure given by the vanishing
cohomology of the fibers, viewed as an element of $\MHM \bigl( \Psm \bigr)$. Finally,
$\jlsl$ is the intermediate extension functor.

The methods of this paper produce a stronger vanishing theorem for the modules
$E_{0,j}$, namely that $E_{0,j} = 0$ for every $j \geq 2$ (the proof can be found in
\subsecref{subsec:ample} below). If we combine this small improvement with the
results of \cite{BFNP}, we obtain the following useful formulas for the cohomology of
the universal hypersurface. 

\begin{theorem}[Brosnan, Fang, Nie, and Pearlstein] \label{thm:BFNP}
If the vanishing cohomology of the hypersurfaces is nontrivial, then 
\[
	H^k \pil \QfamXH \decal{\dfamX} \simeq \begin{cases}
		H^{k+n-1}(X) \tensor \QPH \decal{d} & \text{for $k < 0$,} \\
		H^{k+n+1}(X)(1) \tensor \QPH \decal{d} & \text{for $k > 0$,} \\
		\jlsl \Vvan \oplus H^{n-1}(X) \tensor \QPH \decal{d} & \text{for $k=0$,}
	\end{cases}
\]
as polarized Hodge modules on $P$ of weight $k + (d+n-1)$.
\end{theorem}

\begin{proof}
This follows from \eqref{eq:decompk} and the other results quoted above.
\end{proof}

\subsection{An exact sequence}

We shall now derive most of the above results from a somewhat different point of
view; this will also show how the $\Dmod$-module $(\Mmod, F)$ is connected to mixed
Hodge modules.
\begin{diagram}[midshaft,l>=2em]
	\famX &\rTo^i& P \times X &\lTo^g& \famU \\
		&\rdTo^{\pi}& \dTo &\ldTo^{q}& \\
	&& P
\end{diagram}
Let $\famU = P \times X \setminus \famX$ be the complement of the universal hypersurface.
We denote by $g$ the inclusion of $\famU$ into $P \times X$, and by $i$ that of $\famX$ into $P
\times X$, as shown in the diagram.  Our starting point is the distinguished triangle 
\disttrianglelinelong{}%
	{\il \ius \QPXH}%
	{\QPXH}%
	{\gl \gu \QPXH}%
	{\il \ius \QPXH \decal{1}}{}{}{}
in the category $D^b \MHM(P \times X)$ \cite{SaitoMHM}*{(4.4.1) on p.~321}. 
Obviously, $\gu \QPXH = \QUH$; moreover, since $\famX \subseteq P \times X$ is a
smooth hypersurface, Verdier duality gives an isomorphism
$\ius \QPXH \simeq \QfamXH(-1) \decal{-2}$.
We can therefore rotate the triangle one step to the left and shift by $\dPX$ steps
to obtain
\disttrianglelinelong{}%
	{\QPXH \decal{\dPX}}%
	{\gl \QUH \decal{\dfamU}}%
	{\il \QfamXH(-1) \decal{\dfamX}}%
	{\QPXH \decal{\dPX + 1}}{}{}{}
We then apply the functor $\prPl$ and take cohomology; this gives a long exact sequence
in $\MHM(P)$, a typical portion of which is (for $k \in \ZZ$ arbitrary)
\begin{equation} \label{eq:MHM-seq}
\begin{diagram}
	H^{k-1} \pil \QfamXH(-1) \decal{\dfamX} \\
	\dTo \\
	H^k \prPl \QPXH \decal{\dPX} &\rTo& H^k \ql \QUH \decal{\dfamU}
		&\rTo& H^k \pil \QfamXH(-1) \decal{\dfamX} \\
	&& && \dTo \\
	&& && H^{k+1} \prPl \QPXH \decal{\dPX}.
\end{diagram}
\end{equation}

Ultimately, we would like to find $H^0 \pil \QfamXH \decal{\dfamX}$; this makes it
necessary to compute the other mixed Hodge modules in the sequence.
Throughout, $H^i(X)$ denotes the Hodge structure on the $i$-th cohomology of
$X$, viewed as an object in $\MHM(\pt)$.

\begin{lemma} \label{lem:QPX}
For each $k \in \ZZ$, we have $H^k \prPl \QPXH \decal{\dPX} \simeq H^{k+n}(X) \tensor \QPH
\decal{d}$.
\end{lemma}

\begin{proof}
By Proper Base Change \cite{SaitoMHM}*{(4.4.3) on p.~323} for the diagram
\begin{diagram}[midshaft,l>=2.5em]
	P \times X &\rTo^{\prX}& X \\
	\dTo_{\prP} & & \dTo_{a_X} \\
	P &\rTo^{a_P}& \pt,
\end{diagram}
we have $\prPl \QPXH = \prPl \prXu \QXH = \aPu \aXl \QXH$. From the decomposition
\[
	\aXl \QXH \simeq \bigoplus_i H^i \aXl \QXH \decal{-i} 
		= \bigoplus_i H^i(X) \decal{-i}
\]
in $\MHM(\pt)$ \cite{SaitoHM}*{Corollaire~3 on p.~857}, it follows that
\[
	\prPl \QPXH \decal{\dPX} \simeq \aPu \bigoplus_i H^i(X)[\dPX-i] = 
		\bigoplus_i H^i(X) \tensor \QPH \decal{\dPX-i}.
\]
Now apply $H^k$ to get $H^k \prPl \QPXH \decal{\dPX} \simeq H^{k+n}(X) \tensor \QPH
\decal{d}$, since $\QPH \decal{\dPX-i}$ sits in degree $d - (\dPX - i) = i - n$.
\end{proof}

The cohomology of a smooth affine variety vanishes above the middle dimension. The
following lemma shows that a similar result is true for the cohomology of $\ql \QUH
\decal{\dfamU}$; given that $U = P \times X \setminus \famX$ is affine, this is not
surprising.

\begin{lemma} \label{lem:QUH-van}
We have $H^k \ql \QUH \decal{\dfamU} = 0$ for all $k > 0$.
\end{lemma}
\begin{proof}
We give a proof using $\Dmod$-modules. The mixed Hodge module $\gl \QUH
\decal{\dfamU}$ has underlying $\Dmod$-module $\OPX(\ast \famX)$, because $\famX$ is
nonsingular \cite{SaitoB}*{Corollary~4.3}.
The $\Dmod$-module associated to $\ql \QUH \decal{\dfamU} = \prPl \gl \QUH
\decal{\dfamU}$ is therefore the direct image of $\OPX(\ast \famX)$; it can be
computed from the relative de Rham complex
\begin{multline*}
	\DR_{P \times X / P} \bigl( \OPX(\ast \famX) \bigr) =  \\
\begin{diagram}
	\Bigl\lbrack 
		\OPX(\ast \famX) &\rTo& \OmPXP{1}(\ast \famX) &\rToDots& \OmPXP{n}(\ast \famX)
	\Bigr\rbrack \decal{n}.
\end{diagram}
\end{multline*}
Noting that each sheaf in the complex is acyclic for the functor $\prPl$, the direct
image is represented by the complex $\prPl \DR_{P \times X / P} \bigl( \OPX(\ast
\famX) \bigr)$. It is clearly supported in degrees $-n, \dotsc, 0$, and therefore
the cohomology sheaf in degree $k > 0$ vanishes. Since the functor that takes a mixed
Hodge module to its underlying $\Dmod$-module is faithful, we conclude that $H^k
\ql \QUH \decal{\dfamU}$ is also zero.
\end{proof}

Finally, we borrow the following lemma from \cite{BFNP}*{Proposition~4.8}.

\begin{lemma}[Brosnan, Fang, Nie, and Pearlstein] \label{lem:BFNP2}
We have
\[
	E_{0,0} \simeq \jlsl V^{n-1} \simeq 
		\jlsl \Vvan \oplus H^{n-1}(X) \tensor \QPH \decal{d}.
\]
\end{lemma}
\begin{proof}
By definition, $E_{0,0}$ is the piece in the decomposition of $H^0 \pil \QfamXH
\decal{\dfamX}$ that has strict support equal to all of $P$. By the Base Change
Theorem, applied to the inclusion $j \colon \Psm \to P$, we have
\[
	\ju H^0 \pil \QfamXH \decal{\dfamX} 
		\simeq H^0 \pisml \QQ_{\famXsm}^H \decal{\dfamX} = 
		H^{n-1} \pisml \QQ_{\famXsm}^H \decal{d} = V^{n-1}.
\]
Therefore, $\jlsl V^{n-1}$ is a submodule of $H^0 \pil \QfamXH \decal{\dfamX}$. Since
all other terms $E_{0,j}$ are supported in proper subvarieties, we conclude that
$\jlsl V^{n-1} \simeq E_{0,0}$. The second isomorphism is then an immediate
consequence of \eqref{eq:V-direct-sum}.
\end{proof}

\subsection{Analysis of the exact sequence}
\label{subsec:pieces}

We can now look at the exact sequence in \eqref{eq:MHM-seq} one more time. For $k >
0$, each portion of the sequence simplifies to
\begin{diagram}
	0 &\rTo& H^k \pil \QfamXH (-1) \decal{\dfamX} &\rTo& 
		H^{n+1+k}(X) \tensor \QPH \decal{d} &\rTo& 0,
\end{diagram}
using the vanishing in Lemma~\ref{lem:QUH-van}, and the result of
Lemma~\ref{lem:QPX}. In terms of the decomposition \eqref{eq:decompk}, we thus have
\[
	E_{k,0} = E_k \simeq H^{n+1+k}(X)(1) \tensor \QPH \decal{d} 
		\quad \text{for $k > 0$,}
\]
noting that $E_{k,j} = 0$ if $j \neq 0$ because $E_k$ is supported on all of $P$.
For $k < 0$, we deduce from this and Saito's Hard Lefschetz Theorem \eqref{eq:SHL}
that
\[
	E_{k,0} \simeq E_{-k,0}(-k) \simeq H^{n+1-k}(X)(1-k) \tensor \QPH \decal{d}
		\simeq H^{n+k-1}(X) \tensor \QPH \decal{d},
\]
where the last isomorphism is because of the usual Hard Lefschetz Theorem. Since both
isomorphisms are induced by the polarization, and therefore compatible, 
it follows that the restriction map $H^{n+k-1}(X) \tensor \QPH \decal{d} \to E_{k,0}$
itself has to be an isomorphism.  Again, $E_{k,j} = 0$ if $k < 0$ and $j \neq
0$, and so we have proved (a).

Next, we look at the exact sequence in negative degrees. After incorporating the
results from above, a typical portion simplifies to
\begin{diagram}[l>=2em]
H^{n+k-2}(X)(-1) \tensor \QPH \decal{d} \\
\dTo \\
H^{n+k}(X) \tensor \QPH \decal{d} &\rTo& H^k \ql \QUH \decal{\dfamU} &\rTo&
H^{n+k-1}(X)(-1) \tensor \QPH \decal{d} \\
&& && \dTo \\
&& && H^{n+k+1}(X) \tensor \QPH \decal{d}.
\end{diagram}
Both vertical maps are injective (by the usual Hard Lefschetz Theorem), and so we
find that
\begin{equation} \label{eq:kneg}
	H^k \ql \QUH \decal{\dfamU} \simeq \frac{H^{n+k}(X)}{H^{n+k-2}(X)(-1)} 
		\tensor \QPH \decal{d}
\end{equation}
when $k < 0$.

Finally, the part of the exact sequence for $k = 0$ reads
\begin{diagram}
H^{n-2}(X)(-1) \tensor \QPH \decal{d} \\
\dTo \\
H^{n}(X) \tensor \QPH \decal{d} &\rTo& H^k \ql \QUH \decal{\dfamU} &\rTo&
H^0 \pil \QfamXH \decal{\dfamX} \\
&& && \dTo \\
&& && H^{n+1}(X) \tensor \QPH \decal{d}.
\end{diagram}
From \eqref{eq:decompk} and Lemma~\ref{lem:BFNP2}, we have the decomposition
\begin{equation} \label{eq:decompH0}
	H^0 \pil \QfamXH \decal{\dfamX}
 	\simeq \jlsl \Vvan \oplus H^{n-1}(X) \tensor \QPH \decal{d} \oplus 
		\bigoplus_{j > 0} E_{0,j}.
\end{equation}
Combining this with the fact that the primitive cohomology of weight $n$ satisfies
$\HXprim \simeq H^n(X) / H^{n-2}(X)(-1)$, we obtain the short exact sequence
\begin{equation} \label{eq:kzero}
\begin{diagram}
	0 &\rTo& \HXprim \tensor \QPH \decal{d} &\rTo& H^0 \ql \QUH \decal{\dfamU}
		&\rTo& \jlsl \Vvan(-1) \oplus R &\rTo& 0
\end{diagram}
\end{equation}
Here $R = \bigoplus_{j > 0} E_{0,j}(-1)$ is a sort of ``error term,'' containing those
pieces in the decomposition of $H^0 \pil \QfamXH \decal{\dfamX}$ whose support is
contained in the dual variety $\dualX$. As mentioned above, it was shown in
\cite{BFNP} that $R = 0$ for sufficiently ample $\OX(1)$; we give a more precise
statement in \subsecref{subsec:ample}.

\subsection{The underlying $\Dmod$-modules}
\label{subsec:underlyingD}

We now pass to the underlying $\Dmod$-modules in the exact sequence
\eqref{eq:MHM-seq}. As explained during the proof of Lemma~\ref{lem:QUH-van},
the mixed Hodge module $\gl \QUH \decal{\dfamU}$ has associated $\Dmod$-module
$\OPX(\ast \famX)$, with filtration
\begin{equation} \label{eq:pole}
	F_k \OPX(\ast \famX) = \begin{cases}
		\OPX \bigl( (k+1) \famX \bigr) & \text{if $k \geq 0$,} \\
		0 & \text{otherwise.}
	\end{cases}
\end{equation}
To compute the direct images under the projection $P \times X \to P$, we use the
relative de Rham complex
\begin{diagram}
	\DR_{P \times X / P} = \Bigl\lbrack 
		\OPX &\rTo& \OmPXP{1} &\rToDots& \OmPXP{n}
	\Bigr\rbrack \decal{n},
\end{diagram}
with differential $\dPXP$. For $q > 0$, we have
\[
	R^q \prPl \OmPXP{k} \tensor \OPX(\ast \famX) = 
		R^q \prPl \OmPXP{k}(\ast \famX) = 0,
\] 
and so each of the sheaves $\OmPXP{k}(\ast \famX)$ is acyclic for the push-forward
map $\prPl$. Thus the complex $\prPl \DR_{P \times X / P} \bigl( \OPX(\ast \famX)
\bigr)$, which looks like
\begin{equation} \label{eq:dRpr}
\begin{diagram}
 	\Bigl\lbrack 
		\prPl \OPX(\ast \famX) &\rTo& \prPl\OmPXP{1}(\ast \famX) &\rToDots& 
			\prPl \OmPXP{n}(\ast \famX) \Bigr\rbrack \decal{n},
\end{diagram}
\end{equation}
represents the direct image of $\OPX(\ast \famX)$ in the derived category of
filtered holonomic complexes on $P$. Note that each term in the complex is naturally
a $\Dmod$-module on $P$; moreover, the maps in the complex are $\OP$-linear.
We conclude that the $j$-th cohomology sheaf $\Nmod^j$ of the complex,
\[
	\Nmod^j = H^j \prPl \DR_{P \times X / P} \bigl( \OPX(\ast \famX) \bigr),
\]
is the holonomic $\Dmod$-module underlying the mixed Hodge module $H^j \ql \QUH
\decal{\dfamU} = H^j \prPl \gl \QUH \decal{\dfamU}$. As in Lemma~\ref{lem:QUH-van},
$\Nmod^j = 0$ outside the range $-n \leq j \leq 0$.

Both the complex in \eqref{eq:dRpr} and its cohomology sheaves $\Nmod^j$ are
naturally filtered. Indeed, the pole-order filtration on $\OPX(\ast \famX)$
defines a filtration on the relative de Rham complex by subcomplexes $F_k \DR_{P
\times X / P} \bigl( \OPX(\ast \famX) \bigr)$, equal to
\begin{diagram}
	\Bigl\lbrack F_k &\rTo& \OmPXP{1} \tensor F_{k+1} &\rToDots& \OmPXP{n} \tensor
		F_{k+n} \Bigr\rbrack \decal{n},
\end{diagram}
where $F_k = F_k \OPX(\ast \famX)$ is as in \eqref{eq:pole}. The filtration on the
complex in \eqref{eq:dRpr} is then induced by the subcomplex $\derR \prPl F_k
\DR_{P \times X / P} \bigl( \OPX(\ast \famX) \bigr)$; note that $\prPl$ means the
derived functor here, as the individual sheaves in the complex are not necessarily
acyclic. Since the morphism $\prP$ is projective, the filtration is strict by a
result of Saito's \cite{SaitoMHM}*{Theorem~2.14}; this means that the cohomology
sheaves of the subcomplex inject into those of the whole complex. The filtration on
each of the $\Dmod$-modules $\Nmod^j$ is the induced filtration,
\[
	F_k \Nmod^j = H^j \prPl F_k \DR_{P \times X / P} \bigl( \OPX(\ast \famX) \bigr),
\]
where $\prPl$ is again the derived functor. Observe that each filtered $\Dmod$-module
$\bigl( \Nmod^j, F \bigr)$ is regular holonomic, because it underlies a mixed Hodge module.

\begin{proposition} \label{prop:Dmod-Nj}
Let $F_k = F_k \OPX(\ast \famX)$, and consider the subcomplex
\begin{equation} \label{eq:dRprsub}
\begin{diagram}
 	\Bigl\lbrack 
		\prPl F_k &\rTo& \prPl \OmPXP{1} \tensor F_{k+1}  &\rToDots& 
			\prPl \OmPXP{n} \tensor F_{k+n} \Bigr\rbrack \decal{n}
\end{diagram}
\end{equation}
of the complex in \eqref{eq:dRpr}. Then for $k \gg 0$, the coherent sheaf $F_k
\Nmod^j$ is the $j$-th cohomology sheaf of the complex in \eqref{eq:dRprsub}. The
same is true for arbitrary $k \in \ZZ$, provided that the line bundle $\OX(1)$
satisfies the condition
\begin{equation} \label{eq:ample}
	\text{$H^q \bigl( X, \OmX{p}(k) \bigr) = 0$ for $q > 0$ and $k > 0$}.
\end{equation}
\end{proposition}

\begin{proof}
For $k \gg 0$, each sheaf in the complex is acyclic for the functor $\prPl$. The same
is true for arbitrary $k \in \ZZ$ under the displayed condition on $\OX(1)$.
\end{proof}

The most interesting among those $\Dmod$-modules is 
\begin{equation} \label{eq:N0}
	\Nmod^0 = \frac{\prPl \OmPXP{n}(\ast \famX)}
		{\dPXP \Bigl( \prPl \OmPXP{n-1}(\ast \famX) \Bigr)},
\end{equation}
which is filtered by the coherent subsheaves $F_k \Nmod^0$. For $k \gg 0$, or in
general if the condition in \eqref{eq:ample} is satisfied, these sheaves are given by
the formula
\begin{equation} \label{eq:HFN}
	F_k \Nmod^0 = \frac{\prPl \OmPXP{n} \bigl( (n+k+1) \famX \bigr)}
		{\dPXP \Bigl( \prPl \OmPXP{n-1} \bigl( (n+k) \famX) \bigr) \Bigr)}.
\end{equation}

Now let $(\Mvan, F)$ be the filtered $\Dmod$-module underlying $\jlsl \Vvan$; it is
the minimal extension of the flat vector bundle $\OPsm \tensor_{\CC} \Rvan{n-1}
\pisml \CC$. Temporarily, we also introduce the filtered $\Dmod$-module $(\Rmod, F)$, 
underlying the error term $R = \bigoplus_{j > 0} E_{0,j}$ in \eqref{eq:kzero}.

The following theorem summarizes the results of this and the previous section.
To explain the notation, we remark that for a mixed Hodge structure $H$,
the $\Dmod$-module $\HC \tensor \OP$ underlying $\aPu H = H \tensor \QPH
\decal{d}$ has as its filtration
\begin{equation} \label{eq:filt}
	F_k \bigl( \HC \tensor \OP \bigr) = \bigl( F_k \HC \bigr) \tensor \OP
			= \bigl( F^{-k} \HC \bigr) \tensor \OP,
\end{equation}
induced from the Hodge filtration on $H$. 

\begin{theorem} \label{thm:Dmod1}
For any $j < 0$, we have an isomorphism of filtered $\Dmod$-modules
\[
	\Nmod^j \simeq \frac{H^{n+j}(X, \CC)}{H^{n+j-2}(X, \CC)} \tensor \OP,
\]
where the right-hand side has the filtration described in \eqref{eq:filt}.
Moreover, we have a short exact sequence of filtered $\Dmod$-modules
\begin{diagram}
	0 &\rTo& \Hprim{n}(X, \CC) \tensor \OP &\rTo& \Nmod^0 &\rTo& \Mvan(-1) \oplus \Rmod
		&\rTo& 0,
\end{diagram}
strict with respect to the filtrations. 
\end{theorem}

\begin{proof}
The first assertion is an immediate consequence of the isomorphism in
\eqref{eq:kneg}. Indeed, we have just shown that the filtered $\Dmod$-module underlying the
left-hand side of \eqref{eq:kneg} is $\Nmod^j$, whereas the one
underlying the right-hand side is evidently $H^{n+j}(X, \CC) \big/ H^{n+j-2}(X, \CC)
\tensor \OP$. The second assertion follows in the same way from the exact sequence in
\eqref{eq:kzero}, recalling that morphisms between mixed Hodge modules strictly
preserve the Hodge filtrations on the underlying $\Dmod$-modules.
\end{proof}

In \subsecref{subsec:ample}, we show $R = 0$ if and only if the vanishing cohomology
of the hypersurfaces is nontrivial; this means that precisely one of $\Rmod$ and
$\Mvan$ is nonzero. After incorporating this result, we arrive at the following
description of the filtered $\Dmod$-module $\Mmod$.

\begin{corollary} \label{cor:Dmod2}
Suppose that the vanishing cohomology of the hypersurfaces is nontrivial; then we
have $\Mvan = \Mmod$. If moreover the condition in \eqref{eq:ample} is satisfied,
then $F_k \Mvan = F_{k+n} \Mmod$, and the exact sequence in Theorem~\ref{thm:Dmod1} becomes
\begin{diagram}
	0 &\rTo& \Hprim{n}(X, \CC) \tensor \OP &\rTo& \Nmod^0 &\rTo& \Mmod(-n-1) &\rTo& 0.
\end{diagram}
\end{corollary}

\begin{proof}
By construction, the restriction of $\Mvan$ to $\Psm$ is the flat vector bundle
$\shHO$, and so $\Mvan$ is a subsheaf of $\jl \shHO$ by the definition of the
intermediate extension functor. Since $\Rmod = 0$ by Proposition~\ref{prop:R}, the
$\Dmod$-module $\Mvan$ is a quotient of $\Nmod^0$, and over $\Psm$, the resulting map
from $\Nmod^0 \restr{\Psm}$ to $\shHO$ is clearly the fiberwise residue map. We
conclude that $\Mvan = \Mmod$. As for the filtrations, note that $F_k \Mvan$ is a quotient of
$F_{k-1} \Nmod^0$, and hence of $\prPl \OmPXP{n} \bigl( (k+n) \famX \bigr)$.
\end{proof}

\subsection{Vanishing of higher cohomology}
\label{subsec:vanishing}

In this section, we use the direct image of the relative de Rham complex in
\eqref{eq:dRpr} to construct something like a locally free resolution for the sheaves
$F_k \Nmod^0$, and then use it to prove a vanishing theorem for their higher
cohomology groups. Fix $k \in \ZZ$; we assume either that \eqref{eq:ample} is
satisfied, or that $k \gg 0$. Consider the subcomplex 
\[	
	\Emod_k^{\bullet} = \prPl F_{k-n-1} \DR_{P \times X / P} \bigl( \OPX(\ast \famX) \bigr)
\]
of the direct image of the relative de Rham complex as in \eqref{eq:dRprsub}. It has
$\OP$-linear differentials, and is supported in degrees $-n, \dotsc, 0$.
The individual sheaves in the complex are easily found to be
\begin{align*}
	\Emod_k^i &= \prPl \Bigl( \OmPXP{n+i} \tensor F_{k+i-1} \OPX(\ast \famX) \Bigr) \\
		&=	\begin{cases}
		H^0 \bigl( X, \OmX{n+i}(k+i) \bigr) \tensor 
			\OP(k+i) & \text{if $k+i \geq 1$,} \\
		0 & \text{otherwise.}
	\end{cases}
\end{align*}
The discussion in \subsecref{subsec:underlyingD} shows that the cohomology sheaves of
the complex are precisely the coherent sheaves $F_{k-n-1} \Nmod^i$.
Theorem~\ref{thm:Dmod1} describes these sheaves for $i < 0$,
\begin{equation} \label{eq:FsNp}
	F_{k-n-1} \Nmod^i \simeq
		\frac{F^{n+1-k} H^{n+i}(X, \CC)}{F^{n-k} H^{n+i-2}(X, \CC)} \tensor \OP.
\end{equation}
This is almost as good as having a locally free resolution, and easily allows us to
prove the following vanishing theorem.

\begin{theorem} \label{thm:vanishing}
Suppose that the line bundle $\OX(1)$ satisfies the condition in \eqref{eq:ample}, or
that $k \gg 0$. Then we have $H^i \bigl( P, \OmP{p} \tensor F_k \Nmod^0 \bigr) = 0$
for every $i \geq \max(p-1, 0)$.  
\end{theorem}

\begin{proof}
The terms of the complex $\OmP{p} \tensor \Emod_k^{\bullet}$ are
direct sums of sheaves of the form $\OmP{p}(m)$ with $m \geq 1$. In the
first hypercohomology spectral sequence,
\[
	\primeE_1^{i,j} = H^j(P, \OmP{p} \tensor \Emod_k^i) 
		\Longrightarrow \HH^{i+j} \bigl( \OmP{p} \tensor \Emod_k^{\bullet} \bigr),
\]
we have $\primeE_1^{i,j} = 0$ for $j > 0$ by Bott's vanishing theorem. All but one
row is therefore zero, and so the spectral sequence degenerates. Since moreover
$\primeE_1^{i,0} = 0$ if $i > 0$, it follows that $\HH^i \bigl( \OmP{p} \tensor
\Emod_k^{\bullet} \bigr) = 0$ whenever $i > 0$.

To relate this to the cohomology of $F_{k-n-1} \Nmod^0$, we use the second
hypercohomology spectral sequence,
\[
	\pprimeE_2^{i,j} = H^i \bigl( P, \OmP{p} \tensor F_{k-n-1} \Nmod^j \bigr) 
		\Longrightarrow \HH^{i+j} \bigl( \OmP{p} \tensor \Emod_k^{\bullet} \bigr),
\]
remembering that $F_{k-n-1} \Nmod^j$ is the $j$-th cohomology sheaf of the complex
$\Emod_k^{\bullet}$. Now $\OmP{p} \tensor F_{k-n-1} \Nmod^j$ is a direct sum of
copies of $\OmP{p}$ for $j < 0$, and therefore $\pprimeE_2^{i,j} = 0$ unless $i = p$
or $j = 0$.  From the vanishing of the hypercohomology in positive degrees, we now conclude that
\[
	H^i \bigl( P, \OmP{p} \tensor F_{k-n-1} \Nmod^0 \bigr) = 
		\pprimeE_2^{i,0} = 0
\]
for $i \geq \max(p-1, 0)$, as asserted.
\end{proof}

On the other hand, Corollary~\ref{cor:Dmod2} relates $\Nmod^0$ and $\Mmod$ through
the short exact sequence
\begin{diagram}
	0 &\rTo& F^{n+1-k} \Hprim{n}(X, \CC) \tensor \OP &\rTo& F_{k-n-1} \Nmod^0 &\rTo& 
		F_k \Mmod &\rTo& 0.
\end{diagram}
The following vanishing theorem for $F_k \Mmod$ is an immediate consequence (note if
the vanishing cohomology of the hypersurfaces is trivial, then $\Mmod = 0$, and so
the result holds by default).

\begin{corollary} \label{cor:vanishing}
If the condition in \eqref{eq:ample} is satisfied, then
$H^i \bigl( P, \OmP{p} \tensor F_k \Mmod) = 0$ for every $i \geq \max(p,0)$.
\end{corollary}

The same spectral sequence argument can also be used to compute spaces of global
sections. To state the result, we introduce the notation
\[
	W_k^p = H^0 \bigl( P \times X, \OmPXP{p}(k \famX) \bigr) \simeq
		H^0 \bigl( X, \OmX{p}(k) \bigr) \tensor H^0 \bigl( P, \OP(k) \bigr),
\]
noting that the relative differential $\dPXP$ maps $W_k^p$ to $W_{k+1}^{p+1}$.

\begin{corollary} \label{cor:glob-sect}
For $k \gg 0$, or for any $k \geq 1$ assuming \eqref{eq:ample}, we have
\[
	H^0 \bigl( P, F_{k-n-1} \Nmod^0 \bigr) \simeq
		\frac{W_k^n}{\dPXP W_{k-1}^{n-1}}.
\]
If the line bundle $\OX(1)$ is sufficiently ample, then similarly
\[
	H^0 \bigl( P, F_k \Mmod \bigr) \simeq
		\frac{W_k^n}{\dPXP W_{k-1}^{n-1} + F^{n+1-k} \Hprim{n}(X, \CC)}.
\]
\end{corollary}

\begin{corollary}
If $\OX(1)$ satisfies the condition in \eqref{eq:ample}, then $F_k \Mmod$ is a
quotient of the ample vector bundle $H^0 \bigl( X, \OmX{n}(k) \bigr) \tensor \OP(k)$,
and therefore globally generated.
\end{corollary}

\begin{proof}
$F_k \Mmod$ is a quotient of $F_{k-n-1} \Nmod^0$, which in turn is a quotient of
$\Emod_k^0 = H^0 \bigl( X, \OmX{n}(k) \bigr) \tensor \OP(k)$, at least when $k \geq
1$. This proves the assertion, because $F_k \Mmod = 0$ for $k \leq 0$.
\end{proof}

\subsection{The vanishing of the error term}
\label{subsec:ample}

Recall that the error term $R$ in \eqref{eq:kzero} is the direct sum of the
Hodge modules $E_{0,j}$ with $j > 0$. In \cite{BFNP}, it was shown that $R = 0$ once
the line bundle $\OX(1)$ is sufficiently ample, and that $E_{0,1} = 0$ precisely when
the vanishing cohomology of the hypersurfaces is nontrivial. We improve this to a
necessary and sufficient condition for the vanishing of the Hodge module $R$.

\begin{proposition} \label{prop:R}
We have $E_{0,j} = 0$ for every $j \geq 2$. If the vanishing cohomology of the
hypersurfaces is nontrivial, then $E_{0,j} = 0$ for every $j > 0$; in particular, the
error term $R = 0$ in \eqref{eq:kzero} is then zero.
\end{proposition}

\begin{note}
By recent work of Dimca and Saito \cite{DS}*{Theorem~6}, it is known that the
vanishing cohomology is nontrivial as soon as $\OX(1)$ is the third (or, in most
cases, second) power of a very ample line bundle.
\end{note}

The proof of the proposition is based on computing the characteristic variety of the
$\Dmod$-module $\Nmod^0$; we will show that it has exactly two components, one of
which is the cone over the set of singularities $\famY \subseteq \famX$, and
determine the multiplicities. Our starting point is the following simple lemma.

\begin{lemma} \label{lem:Koszul}
As usual, let $\phi \colon \famX \to X$ and $\psi \colon \famY \to X$ be the second
projections. For every $k \in \ZZ$, the Koszul-type complex
\begin{diagram}
	\OfamX(k-n) &\rTo& \phiu \OmX{1} \tensor \OfamX(k-n+1) &\rToDots&
		\phiu \OmX{n} \tensor \OfamX(k).
\end{diagram}
is a locally free resolution on $\famX$ for the sheaf $\psiu \OmX{n} \tensor
\OfamY(k)$.
The differential in this complex
is given by the rule $\beta \mapsto \dPXP \sfamX \wedge \beta$, where $\sfamX$
is the section of $\OPX(1)$ defining $\famX \subseteq P \times X$.
\end{lemma}

\begin{proof}
From the embedding $i \colon X \into Q$ given by the linear system of $\OX(1)$, we
have the following exact sequence:
\begin{diagram}
0 &\rTo& \Theta_X &\rTo& \iu \Theta_Q &\rTo& \NXQ &\rTo& 0.
\end{diagram}
The projectivization $\PP(\iu \Theta_Q)$ is naturally isomorphic to $\famX$, and the
exact sequence determines a canonical section $\alfamX$ of the vector bundle $(\phiu
\Theta_X)^{\vee} \tensor \shO_{\PP(\iu \Theta_Q)}(1) \simeq \phiu \OmX{1} \tensor
\OfamX(1)$, whose zero scheme is $\PP(\NXQ) \simeq \famY$. As a consequence, $\famY$
is a locally complete intersection in $\famX$, and its structure sheaf is resolved by
the Koszul-type complex
\begin{diagram}
	\OfamX &\rTo& \phiu \OmX{1} \tensor \OfamX(1) &\rToDots&
		\phiu \OmX{n-1} \tensor \OfamX(n-1) &\rTo&
		\phiu \OmX{n} \tensor \OfamX(n),
\end{diagram}
whose differential is given by wedge product with $\alfamX$. It is then not hard to
see that $\alfamX = \dPXP \sfamX$, viewed as a map
\[
	\dPXP \sfamX \colon \OfamX(-1) \to \OmPXP{1} \big\vert_{\famX} 
		\simeq \phiu \OmX{1}.
\]
We obtain the asserted resolution after tensoring by $\OfamX(k-n)$.
\end{proof}

We are now in a position to compute the coherent sheaf on $\PP(\Theta_P)$ associated
to the graded module $\Gr^F \Nmod^0$.

\begin{lemma} \label{lem:charN0}
We have $\shC \bigl( \Nmod^0, F \bigr) \simeq \psiu \OmX{n} \tensor \OfamY(n+1)$,
viewed as a coherent sheaf on $\famY \into \PP(\Theta_P)$ under the natural inclusion
map.
\end{lemma}

\begin{proof}
As in Lemma~\ref{lem:charN}, we shall compute the characteristic module $C \bigl(
\Nmod^0, F \bigr)$ directly, at least in sufficiently high degrees $k \gg 0$. To
simplify the notation, we again use the notation
\[
	W_k^p = H^0 \bigl( P \times X, \OmPXP{p}(k \famX) \bigr) \simeq
		H^0 \bigl( X, \OmX{p}(k) \bigr) \tensor H^0 \bigl( P, \OP(k) \bigr)
\]
for the space of sections of the indicated vector bundle. By definition,
\[
	C \bigl( \Nmod^0, F \bigr) = \bigoplus_{k \in \ZZ} 
		H^0 \bigl( P, F_k \Nmod^0 / F_{k-1} \Nmod^0 \bigr),
\]
and since $H^1 \bigl( P, F_k \Nmod^0 \bigr) = 0$ for $k \gg 0$ by
Theorem~\ref{thm:vanishing}, we get that
\[
	H^0 \bigl( P, F_k \Nmod^0 \big/ F_{k-1} \Nmod^0 \bigr) \simeq 
		\frac{H^0 \bigr( P, F_k \Nmod^0 \bigl)}{H^0 \bigr( P, F_{k-1} \Nmod^0 \bigr)}.
\]
By Corollary~\ref{cor:glob-sect}, $H^0 \bigl( P, F_k \Nmod^0 \bigr) \simeq
W_{k+n+1}^n \big/ \dPXP W_{k+n}^{n-1}$, and so we finally have
\begin{equation} \label{eq:aux2}
	H^0 \bigl( P, F_k \Nmod^0 \big/ F_{k-1} \Nmod^0 \bigr) \simeq
		\frac{W_{k+n+1}^n}{\dPXP W_{k+n}^{n-1} + W_{k+n}^n}.
\end{equation}

Next, we determine the graded module corresponding to $\psiu \OmX{n}$, given by
\[
	\bigoplus_{k \in \ZZ}
		H^0 \bigl( \famX, \psiu \OmX{n} \tensor \OfamX(k) \bigr) \simeq
	\bigoplus_{k \in \ZZ}
		H^0 \bigl( \famY, \psiu \OmX{n} \tensor \OfamY(k) \bigr).
\]
From Lemma~\ref{lem:Koszul}, we know that the sequence
\begin{diagram}
\phiu \OmX{n-1} \tensor \OfamX(k-1) &\rTo& \phiu \OmX{n} \tensor \OfamX(k) &\rTo&
	\psiu \OmX{n} \tensor \OfamY(k) &\rTo& 0
\end{diagram}
is exact; as $\OfamX(1)$ is ample, the corresponding sequence of global sections
remains exact for $k \gg 0$. We therefore obtain 
\[
	H^0 \bigl( \famY, \psiu \OmX{n} \tensor \OfamY(k) \bigr) \simeq
		\frac{H^0 \bigl( \famX, \phiu \OmX{n} \tensor \OfamX(k) \bigr)}%
			{H^0 \bigl( \famX, \phiu \OmX{n-1} \tensor \OfamX(k-1) \bigr)}.
\]
When $k \gg 0$, one easily shows that $H^0 \bigl( \famX, \phiu \OmX{p} \tensor
\OfamX(k) \bigr) \simeq W_k^p / W_{k-1}^p$; since the differential in the exact
sequence is given by $\dPXP \sfamX$, we conclude that
\begin{equation} \label{eq:aux1}
	H^0 \bigl( \famX, \psiu \OmX{n} \tensor \OfamX(k) \bigr) \simeq
		\frac{W_k^n}{\dPXP W_{k-1}^{n-1} + W_{k-1}^n}
\end{equation}
once $k$ is sufficiently large.
The assertion now follows easily upon comparing the expressions in \eqref{eq:aux2}
and \eqref{eq:aux1}.
\end{proof}

Now we turn to the proof of Proposition~\ref{prop:R}.
\begin{proof}
By Lemma~\ref{lem:charN0}, the characteristic sheaf $\shC = \shC \bigl( \Nmod^0, F \bigr)$
is supported on $\famY$ and has rank one. Since the restriction of $\Nmod^0$ to
$\Psm$ is locally free, it follows that the characteristic variety of
$\Nmod^0$ has two irreducible components:
\begin{enumerate}
\item The zero section of $T_P^{\ast}$, with multiplicity equal to the generic rank
of $\Nmod^0$.
\item The cone over the set $\famY$, with multiplicity one.
\end{enumerate}
The same is then true (with a different multiplicity for the first component) for the
holonomic $\Dmod$-module $\Mvan \oplus \Rmod$, since it
differs from $\Nmod^0$ only by the vector bundle $\Hprim{n}(X, \CC) \tensor \OP$ by
Theorem~\ref{thm:Dmod1}. It follows that $E_{0,j} = 0$ for $j \geq 2$: indeed, the
support of $E_{0,j}$ has codimension $j$ in $P$, and so its characteristic variety
could not be contained in that of $\Nmod^0$ unless $j = 1$.

Now suppose that the vanishing cohomology of the hypersurfaces is nontrivial.
Multiplicity is an additive function on holonomic $\Dmod$-modules, and so one of
$\Mvan$ and $\Rmod$ has to have multiplicity zero along the cone over $\famY$. Since
$\Rmod$ is already supported inside $\dualX$, we conclude that if $\Rmod \neq 0$, the
characteristic variety of $\Mvan$ has to consist of just the zero section, which
means that $\Mvan$ has to be a locally free sheaf.
But this can only happen when $\Mvan = 0$, because the monodromy action on the
vanishing cohomology is irreducible \cite{Voisin}*{Corollaire~15.28}.
Indeed, if the $\Dmod$-module $\Mvan$ was locally free, it would be a flat vector
bundle, and therefore trivial (because $P$ is simply connected). In particular, the
local system $\Rvan{n-1} \pisml \CC$ would be trivial. But since it is known that
$H^0 \bigl( \Psm, \Rvan{n-1} \pisml \CC \bigr) = 0$, this is not possible unless
$\Rvan{n-1} \pisml \CC = 0$. Our assumptions rule out this possibility,
and we conclude that $\Mvan \neq 0$ and hence $\Rmod = 0$.
\end{proof}

\begin{example}
It is illustrative to compare our result with the example of plane conics, given
in \ocite{BFNP}*{Example~5.13}. This is the special case when $X = \PPn{2}$, with
line bundle $\shO_{\PPn{2}}(2)$. As the authors remark, $E_{0,0} = 0$, and $E_{0,j} =
0$ for all $j \geq 2$, but $E_{0,1} \neq 0$. In other words, $\Mvan = 0$ while
$\Rmod \neq 0$; note that only one of the summands of $\Rmod$ is nonzero, as required
by the multiplicity calculation above.
\end{example}

We conclude with a description of the characteristic sheaf of $(\Mmod, F)$. 

\begin{corollary} \label{cor:charM}
If the vanishing cohomology of the hypersurfaces is nontrivial, we have $\shC(\Mmod,
F) \simeq \psiu \OmX{n}$.
\end{corollary}

\begin{proof}
This follows almost immediately from the calculations just done. Indeed,
consider the short exact sequence in Corollary~\ref{cor:Dmod2}. 
It shows that $\Mmod$ and $\Nmod^0$ differ only by the locally free sheaf $\Hprim{n}(X, \CC)
\tensor \OP$, and so we have the isomorphism 
\[
	\Gr_k^F \Mmod \simeq \Gr_{k-n-1}^F \Nmod^0;
\]
it holds for any $k > 0$ provided that \eqref{eq:ample} is satisfied, and in general
for $k \gg 0$, which is sufficient here. We then get $\shC(\Mmod, F) \simeq \psiu \OmX{n}$ from
Lemma~\ref{lem:charN0}.
\end{proof}

\subsection{Hypercohomology of the de Rham complex}
\label{subsec:hypercohomology}

In this section, we assume that the vanishing cohomology of the hypersurfaces is
nontrivial, so that $\Mvan = \Mmod$ by Corollary~\ref{cor:Dmod2}.
The $\Dmod$-module $\Mmod$ then underlies the Hodge module $\jlsl
\Vvan$ on $P$. On the other hand, the corresponding perverse sheaf is
\[
	\rat \jlsl \Vvan = \jlsl \rat \Vvan = \jlsl \Rvan{n-1} \pisml \QQ \decal{d};
\]
after tensoring with $\CC$, it becomes isomorphic to the de Rham complex for
$\Mmod$, by the definition of mixed Hodge modules. Therefore,
\begin{equation} \label{eq:iso-dR-perv}
	\DR_P \bigl( \Mmod \bigr) \simeq \jlsl \rat \Vvan \tensor_{\QQ} \CC.
\end{equation}
The purpose of this section is to prove that the hypercohomology $\HH^{-d+1}
\bigl( \DR_P \Mmod \bigr)$ of the de Rham complex is isomorphic to the primitive
cohomology of $X$.

\begin{lemma} \label{lem:unique}
Assume that the vanishing cohomology of the hypersurfaces is nontrivial. Then the
Leray spectral sequence gives rise to a (canonical) isomorphism
\[
	\Hprim{n}(X, \CC) \simeq \HH^{-d+1} \bigl( \DR_P(\Mmod) \bigr).
\]
\end{lemma}

\begin{proof}
Since $\Rmod = 0$, we get $\Mmod = \Mvan$. According to the results in
\subsecref{subsec:pieces}, we have 
\begin{equation} \label{eq:Hkpil}
	H^q \pil \QfamXH \decal{\dfamX} \simeq \begin{cases}
		H^{n+q-1}(X) \tensor \QPH \decal{d} &\text{for $q<0$,} \\
		H^{n+q+1}(X)(1) \tensor \QPH \decal{d} &\text{for $q>0$,} \\
		\jlsl \Vvan \oplus H^{n-1}(X) \tensor \QPH \decal{d} &\text{for $q=0$.}
	\end{cases}
\end{equation}

The $\Dmod$-module component of $\jlsl \Vvan$ is precisely $\Mmod$, and so the
hypercohomology of $\DR_P(\Mmod)$ computes the complex vector spaces underlying the
cohomology modules of $\jlsl \Vvan$. We calculate that
\begin{align*}
	H^{-d+1} \aPl H^0 \pil \QfamXH \decal{\dfamX}  
		&\simeq H^{-d+1} \aPl \jlsl \Vvan \oplus H^{n-1}(X) \tensor H^1 \aPl \QPH \\
		&= H^{-d+1} \aPl \jlsl \Vvan,
\end{align*}
because $H^1 \aPl \QPH = H^1(P) = 0$. It follows that $\HH^{-d+1} \bigl(
\DR_P(\Mmod) \bigr)$ is the complex vector space of the mixed Hodge structure
$H^{-d+1} \aPl H^0 \pil \QfamXH \decal{\dfamX}$.

Now we bring in the (perverse) Leray spectral sequence,
\[
	E_2^{p,q} = H^p \aPl H^q \pil \QfamXH \decal{\dfamX} \Longrightarrow
		H^{p+q} \afamXl \QfamXH \decal{\dfamX} = H^{p+q+\dfamX}(\famX),
\]
which is a spectral sequence of mixed Hodge modules. Note that $E_2^{p,q} = 0$
whenever $p < -d = - \dim P$. The term we are really interested in is 
\[
	E_2^{-d+1,0} \simeq H^{-d+1} \aPl \jlsl \Vvan;
\]
it sits in degree $-d+1$, and is thus a graded quotient of $H^{-d+1+\dfamX}(\famX) =
H^n(\famX)$. 

The Decomposition Theorem implies that the spectral sequence degenerates at the
$E_2$-page (in fact, it even implies that $H^k \afamXl \QfamXH \decal{\dfamX}$ is
isomorphic to the direct sum of all the $E_2^{p,q}$ with $p+q=k$, albeit
non-canonically). Let us write $L^{\bullet}$ for the induced filtration on the
cohomology of $\famX$.  We then have a short exact sequence of mixed Hodge structures
\begin{diagram}[l>=2em]
0 &\rTo& L^1 H^n(\famX) &\rTo& H^n(\famX) &\rTo& E_2^{-d,1} &\rTo& 0,
\end{diagram}
and $E_2^{-d,1} \simeq H^{n+2}(X)(1) \tensor H^0(P)$ by virtue of \eqref{eq:Hkpil}.

Consider now the pullback map $\phiu \colon H^n(X) \to H^n(\famX)$. As the primitive
cohomology is the kernel of $H^n(X) \to H^{n+2}(X)(1)$, we get an induced map from
$H^n(X) \prim$ to $L^1 H^n(\famX)$. The next step of the Leray filtration gives
another short exact sequence
\begin{diagram}[l>=2em]
0 &\rTo& L^2 H^n(\famX) &\rTo& L^1 H^n(\famX) &\rTo^{\rho} & E_2^{-d+1,0} &\rTo& 0,
\end{diagram}
and by composition, we finally obtain a (canonical) map of mixed Hodge structures
\begin{equation} \label{eq:Leray-map}
	\rho \circ \phiu \colon \Hprim{n}(X) \to E_2^{-d+1,0} \simeq 
		H^{-d+1} \aPl \jlsl \Vvan.
\end{equation}

That this map is an isomorphism follows easily from the fact that $\phi \colon \famX
\to X$ is a projective space bundle of rank $d-1$. Indeed, we naturally have
\[
	H^n(\famX) \simeq \bigoplus_{i \geq 0} H^i(P) \tensor H^{n-i}(X)
		= H^n(X) \oplus \bigoplus_{i \geq 2} H^i(P) \tensor H^{n-i}(X).
\]
The terms in the direct sum are precisely the graded quotients of $L^2 H^n(\famX)$,
because the isomorphisms in \eqref{eq:Hkpil} imply that
\[
	E_2^{-d+i,-i+1} = H^{-d+i} \aPl H^{-i-1} \pil \QfamXH \decal{\dfamX}
		\simeq H^i(P) \tensor H^{n-i}(X)
\]
whenever $i \geq 2$. Therefore, the map from $H^n(X)$ to $H^n(\famX) \big/ L^2
H^n(\famX)$ is an isomorphism for dimension reasons; this implies that
\eqref{eq:Leray-map} is also an isomorphism. By passing to the underlying complex
vector spaces, we get the result.
\end{proof}

\subsection{Cohomology sheaves of the de Rham complex}

We conclude our application of the theory of mixed Hodge modules by determining the
cohomology sheaves of the de Rham complex $\DR_P(\Mmod)$. Again, we assume throughout
that the vanishing cohomology of the hypersurfaces is nontrivial, so that $\Mvan =
\Mmod$. Recall that $\DR_P(\Mmod)$ is a perverse complex, because $\Mmod$ underlies a
Hodge module; all of its cohomology sheaves
\[
	\cohshf^k = \cohshf^k \DR_P \bigl( \Mvan \bigr) \simeq \cohshf^k \DR_P \bigl( \Mmod \bigr)
\]
are therefore constructible sheaves (in the Zariski topology). The following lemma
describes them very precisely.

\begin{lemma} \label{lem:cohdeRham}
If the vanishing cohomology of the hypersurfaces is nontrivial, then
\[
	R^{n - 1 + (d + k)} \pil \CC_{\famX} \simeq \cohshf^k \DR_P \bigl( \Mmod \bigr) 
		\oplus H^{n - 1 - (d + k)}(X, \CC) \tensor \CC_P
\]
for all $k \geq -d$.
\end{lemma}

\begin{proof}
The proof mirrors that of Corollary~5.6 in \ocite{BFNP}. Let $p \in P$ be an arbitrary
point. Since the map $\pi \colon \famX \to P$ is proper, we have
\[
	\smash[b]{\bigl( R^{n - 1 + d + k} \pil \CC_{\famX} \bigr)}_p 
		= H^{n - 1 + d + k} \bigl( \famX_p, \CC \bigr)
		= \rat H^{n - 1 + d + k} \bigl( \famX_p \bigr) \tensor_{\QQ} \CC
\]
for the stalk of the higher direct image sheaf at $p$, by the Proper Base Change
Theorem from topology.
Because of the decomposition in \eqref{eq:decomp}, we also have
\[
	H^{n - 1 + d + k} \bigl( \famX_p \bigr) = H^{\dfamX + k} \bigl( \famX_p \bigr)
		= H_p^k \bigl( \pil \QfamXH \decal{\dfamX} \bigr) 
		= \bigoplus_{ij} H_p^{k-i} \bigl( E_{i,j} \bigr).
\]
The assumption on the vanishing cohomology implies that $E_{i,j} = 0$ for all $j \neq
0$; see Proposition~\ref{prop:R} and the work in \subsecref{subsec:pieces}. Therefore
\[
	H^{n - 1 + d + k} \bigl( \famX_p \bigr) 
		= \bigoplus_i H_p^{k-i} \bigl( E_{i,0} \bigr).
\]
The remaining terms are now easily computed. On the one hand,
\[
	H_p^k \bigl( E_{0,0} \bigr) \simeq H_p^k \bigl( \jlsl V^{n-1} \bigr) \simeq
			H_p^k \bigl( \jlsl \Vvan \bigr) \oplus
		\begin{cases}
			H^{n-1}(X) & \text{if $k = -d$,} \\
			0					 & \text{otherwise,}
		\end{cases}
\]
from Lemma~\ref{lem:BFNP2}. On the other hand, we have
\[
	E_{i,0} = \begin{cases}
		H^{n+i-1}(X) \tensor \QPH \decal{d} & \text{for $i < 0$,} \\
		H^{n-i-1}(X)(-i) \tensor \QPH \decal{d} & \text{for $i > 0$,}
	\end{cases}
\]
again using the results in \subsecref{subsec:pieces}. Since
$H_p^{k-i} \bigl( \QPH \decal{d} \bigr) = 0$ for $k - i \neq -d$, it then follows that
\[
	H_p^{k-i} \bigl( E_{i,0} \bigr) \simeq 
		\begin{cases}
			H^{n -1 + (d + k)}(X) & \text{if $i = d + k < 0$,} \\
			H^{n - 1 - (d + k)}(X) \bigl(-(d+k)\bigr) & \text{if $i = d + k > 0$,} \\
			0							 & \text{if $i \neq d + k$ and $i \neq 0.$}
		\end{cases}
\]
In conclusion, we have for $k \geq -d$ an isomorphism
\begin{equation} \label{eq:isoXp}
	H^{n - 1 + (d + k)} \bigl( \famX_p \bigr) \simeq
		H_p^k \bigl( \jlsl \Vvan \bigr) 
			\oplus H^{n - 1 - (d + k)}(X) \bigl(-(d+k) \bigr).
\end{equation}
Now apply the functor $\rat$ to this, and note that 
\[
	\rat H_p^k \bigl( \jlsl \Vvan \bigr) \tensor_{\QQ} \CC \simeq \cohshf_p^k
\]
by the comments preceding the statement of the lemma. On stalks, we thus have
\[
	\smash[b]{\bigl( R^{n - 1 + (d + k)} \pil \CC_{\famX} \bigr)}_p  \simeq
		\cohshf_p^k \oplus H^{n - 1 - (d + k)}(X, \CC),
\]
from which the asserted identity follows immediately.
\end{proof}

\begin{note}
It should be noted that the second part of the isomorphism in \eqref{eq:isoXp} can be
described explicitly. Let $L \colon H^k(X) \to H^{k + 2}(X)(1)$ be the
Lefschetz operator associated with the very ample line bundle $\OX(1)$. For each
hypersurface $\famX_p \subseteq X$, smooth or not, there is then naturally a map
\begin{diagram}[l>=2.5em]
	H^{n - 1 - (d + k)}(X) \bigl( -(d + k) \bigr) &\rTo^{L^{d+k}}&
		H^{n - 1 + (d + k)}(X) &\rTo& H^{n - 1 + (d + k)} \bigl( \famX_p \bigr).
\end{diagram}
So \eqref{eq:isoXp} tells us in particular that this map is always injective.
\end{note}

\section{Properties of the Hodge filtration}

\begin{note}
For the remainder of the paper, we assume that the vanishing
cohomology of the hypersurfaces is nontrivial; this guarantees that $\Mvan = \Mmod$.
\end{note}

\subsection{The lowest level in the filtration}
\label{sec:lowest}

Our next result is about the lowest level in the filtration on $\Mmod$, namely the
sheaf $F_1 \Mmod$, and its relationship to the relative canonical bundle of $\famX
\to P$. In this context, the following theorem by Kawamata comes to mind: If $f
\colon Y \to X$ is an algebraic fiber space such that $f$ is smooth over the
complement of a normal crossing divisor $D$, and $R^{\dim Y - \dim X} \fl \CC$ has
unipotent monodromy on $X \setminus D$, then $\fl \shO_Y(K_{Y/X})$ is locally free
and nef, and agrees with the lowest level of the Hodge filtration on the canonical
extension \cite{Kawamata}*{p.~266}. Obviously, neither of the two assumptions is
satisfied in our case, but a similar result holds.

\begin{proposition} \label{prop:lowest}
Assume that $\OX(1)$ is sufficiently ample. Then $F_1 \Mmod$ is an ample vector bundle, and
forms part of a short exact sequence
\begin{equation} \label{eq:SES-lowest}
\begin{diagram}
0 &\rTo& F_1 \Mmod &\rTo& \pil \OfamX \bigl( K_{\famX/P} \bigr) &\rTo&
	H^{n-1,0}(X) \tensor \OP &\rTo& 0.
\end{diagram}
\end{equation}
In particular, $\pil \OfamX \bigl( K_{\famX/P} \bigr)$ is locally free and nef.
\end{proposition}

\begin{proof}
From the locally free resolution introduced in \subsecref{subsec:vanishing}, we get a
short exact sequence
\begin{equation} \label{eq:SES-F1}
\begin{diagram}
0 &\rTo& H^0 \bigl( X, \OmX{n} \bigr) \tensor \OP &\rTo& 
		H^0 \bigl( X, \OmX{n}(1) \bigr) \tensor \OP(1) &\rTo& F_1 \Mmod
	&\rTo& 0.
\end{diagram}
\end{equation}
The first map is pointwise injective, because $H^0 \bigl( X, \OmX{n} \bigr)
\subseteq H^0 \bigl( X, \OmX{n}(\famX_p) \bigr)$ for every $p \in P$, and so the 
quotient $F_1 \Mmod$ is locally free, proving the first assertion.

To establish \eqref{eq:SES-lowest}, we use the fact that $\famX$ is a smooth
hypersurface in the product $P \times X$, with line bundle $\OPX(\famX) = \OPX(1)$.
Thus the relative canonical bundle for $\pi \colon \famX \to P$ is given by
the formula
\[
	\OfamX \bigl( K_{\famX/P} \bigr) \simeq
		\phiu \OmX{n} \tensor \OfamX(1),
\]
where $\phi \colon \famX \to X$ is the second projection. Pushing forward the
exact sequence
\begin{diagram}
	0 &\rTo& \prXu \OmX{n} &\rTo& \prXu \OmX{n} \tensor \OPX(1) &\rTo&
		\phiu \OmX{n} \tensor \OfamX(1) &\rTo& 0
\end{diagram}
and using the vanishing of $H^1 \bigl( X, \OmX{n}(1) \bigr)$, we get a five-term
exact sequence; the second assertion follows by comparing it with the resolution for
$F_1 \Mmod$ in \eqref{eq:SES-F1}.

Now $F_1 \Mmod$ is evidently ample, since it is a quotient of 
$H^0 \bigl( X, \OmX{n}(1) \big) \tensor \OP(1)$. Because of the
short exact sequence in \eqref{eq:SES-lowest}, it is then immediate that the direct
image $\pil \OfamX \bigl( K_{\famX/P} \bigr)$ is both locally free and nef.
\end{proof}

\begin{note}
A similar result is true in certain other cases: in the 
$34$-dimensional family of all cubic threefolds in $\PPn{4}$, for instance, the coherent sheaf $F_2
\Mmod$ is locally free, and in fact isomorphic to $H^0 \bigl( \PPn{4},
\shO_{\PPn{4}}(1) \bigr) \tensor \OP(2)$.
\end{note}


\subsection{Computation of Ext-groups}

The purpose of this section is the computation of the groups $\Ext_P^i(F_k \Mmod,
\OP)$. We begin by looking at the sheaves $F_{k-n-1} \Nmod^0$.

\begin{lemma} \label{lem:duality-N0}
We have
\[
	\Ext_P^i \bigl( F_{k-n-1} \Nmod^0, \OP \bigr) \simeq 
	\begin{cases}
		0 &\text{for $i = 0, 1$,} \\
		\left( \dfrac{F^{n+1-k} H^{n+1-i}(X, \CC)}{F^{n-k} H^{n-1-i}(X, \CC)} \right)^{\vee}
			&\text{for $2 \leq i \leq d-1$.}
	\end{cases}
\]
\end{lemma}

\begin{proof}
We use the locally free complex $\shE_k^{\bullet}$ from \subsecref{subsec:vanishing};
its cohomology sheaf in degree $i$ is the sheaf $F_{k-n-1} \Nmod^i$.
When we apply the functor $\Hom_P(\argbl, \OP)$ to the complex, we obtain two
spectral sequences; the terms of the first one are
\[
	\primeE_1^{p,q} = \Ext_P^q \bigl( \shE_k^{-p}, \OP \bigr) \simeq 
			H^{d-q} \bigl( P, \shE_k^{-p} \tensor \OP(-d-1) \bigr)^{\vee},
\]
by Serre Duality. Now each $\shE_k^{-p}$ is either zero (when $k-p \leq 0$), or a sum
of positive line bundles $\OP(k-p)$, and so $\primeE_1^{p,q}$ can only be nonzero if
$q = d$ and $0 \leq p \leq k-1-d$. The limit of the spectral sequence is therefore zero
whenever $p+q < d$. We conclude that the second spectral sequence, with terms
\[
	\pprimeE_2^{p,q} = \Ext_P^p \bigl( F_{k-n-1} \Nmod^{-q}, \OP \bigr)
\]
converges to zero in degrees $p+q < d$. For $q < 0$, each $F_{k-n-1} \Nmod^{-q}$ is a
trivial bundle, and so $\pprimeE_2^{p,q} = 0$ unless $p=0$ or $q=0$. This implies
that $\pprimeE_2^{0,p-1} \simeq \pprimeE_2^{p,0}$ for $2 \leq p \leq d-1$. Since
$\pprimeE_2^{p,0} = \Ext_P^p \bigl( F_{k-n-1} \Nmod^0, \OP \bigr)$ and
\[
	\pprimeE_2^{0,p-1} = \Hom \bigl( F_{k-n-1} \Nmod^{1-p}, \OP \bigr)
		\simeq \left( \frac{F^{n+1-k} H^{n+1-p}(X, \CC)}{F^{n-k} H^{n-1-p}(X, \CC)}
		\right)^{\vee}
\]
by \eqref{eq:FsNp}, we get the assertion.
\end{proof}

\begin{theorem} \label{thm:Ext-FkM}
For all integers $0 \leq i \leq d-1$, we have
\[
	\Ext_p^i \bigl( F_k \Mmod, \OP \bigr) \simeq
		\left( \frac{F^{n+1-k} H^{n+1-i}(X, \CC)}{F^{n-k} H^{n-1-i}(X, \CC)} \right)^{\vee}.
\]
In particular, $\Hom_P \bigl( F_k \Mmod, \OP \bigr) = 0$, and $\Ext_P^1 \bigl( F_k
\Mmod, \OP \bigr) \simeq \bigl( F^{n+1-k} \Hprim{n}(X, \CC) \bigr)^{\vee}$.
\end{theorem}
\begin{proof}
The short exact sequence in Corollary~\ref{cor:Dmod2} becomes
\begin{equation} \label{eq:SES-for-Ext}
\begin{diagram}[l>=2em]
	0 &\rTo& F^{n+1-k} \Hprim{n}(X, \CC) \tensor \OP &\rTo& F_{k-n-1} \Nmod^0 
		&\rTo& F_k \Mmod &\rTo& 0
\end{diagram}
\end{equation}
after passing to $F_{k-n-1}$. The result now follows from Lemma~\ref{lem:duality-N0}
by looking at the long exact sequence for the functor $\Hom_P(\argbl, \OP)$.
\end{proof}

The proof shows that the isomorphism between $\Ext_P^1 \bigl( F_k
\Mmod, \OP \bigr)$ and the group $\bigl( F^{n+1-k} \Hprim{n}(X, \CC) \bigr)^{\vee}$ is
given by the extension class of the sequence \eqref{eq:SES-for-Ext}. This class is
an element $\eps_k \in \Ext_P^1 \bigl( F_k \Mmod, F^{n+1-k} \Hprim{n}(X, \CC) \tensor
\OP \bigr)$. Any linear map $f \colon F^{n+1-k} \Hprim{n}(X, \CC) \to \CC$ defines a homomorphism
\[
	\fl \colon \Ext_P^1 \bigl( F_k \Mmod, F^{n+1-k} \Hprim{n}(X, \CC) \tensor \OP \bigr)
		\to \Ext_P^1 \bigl( F_k \Mmod, \OP \bigr),
\]
and $\fl(\eps_k)$ is the element corresponding to $f$ under the isomorphism in the
theorem. 

\begin{note}
Likewise, we have $\Ext_P^1 \bigl( \Gr_k^F \Mmod, \OP \bigr) \simeq
\Hprim{n+1-k,k-1}(X)$.
\end{note}

\subsection{Duality theorems}

In this section, we apply the general duality theorem of \cite{mhmduality} to the
filtered $\Dmod$-module $(\Mmod, F)$. Note that $(\Mvan, F)$ underlies a polarized
Hodge module of weight $\dim \famX = d+n-1$ on $P$, and that we have $F_k \Mvan =
F_{k+n} \Mmod$. Since we have already computed the characteristic module
\[
	\shC(\Mmod, F) \simeq \psiu \OmX{n}
\]
in Corollary~\ref{cor:charM}, the following result is a direct consequence of the duality
theorem. As for notation, we let $\psi \colon \famY \to X$ and $\pi \colon \famY \to P$
denote the two projection maps, and set $\shC = \shC(\Mmod, F)$ and $\shG_k = \Gr_k^F
\Mmod$ for $k \in \ZZ$.

\begin{theorem}
For every $k \in \ZZ$, we have an exact sequence
\begin{equation} \label{eq:dual1}
\begin{diagram}[l>=1.6em]
\shHom(\shG_{n+1-k}, \OP) &\rIntoBold& \shG_k &\rTo& \pil \bigl( \shC \tensor
\OfamY(k) \bigr) &\rOnto& \shExt^1(\shG_{n+1-k}, \OP),
\end{diagram}
\end{equation}
as well as isomorphisms
\begin{equation} \label{eq:dual2}
	R^i \pil \bigl( \shC \tensor \OfamY(k) \bigr) 
		\simeq \shExt^{i+1}(\shG_{n+1-k}, \OP)
\end{equation}
for every $i \geq 1$.
\end{theorem}

The proof in \cite{mhmduality} is based on the fact that $\Gr^F \! \Mmod$ is
Cohen-Macaulay as a graded module over the symmetric algebra $\Sym \Theta_P$. It is
also possible to derive the theorem from a careful analysis of the spectral sequence
\[
	E_1^{p,q} = R^q \pil \bigl( \phiu \OmX{p}(k-n+p) \bigl) \tensor \OP(k-n+p)		
		\Longrightarrow R^{p+q-n} \pil \big( \psiu \OmX{n} \tensor \OfamY(k) \bigr),
\]
coming from the resolution in Lemma~\ref{lem:Koszul}; details can be found in
\cite{thesis}. We note that the first map in \eqref{eq:dual1} is induced by the 
intersection pairing: its restriction to $\Psm$ is a map of vector bundles
\[
	\bigl( F^{k-1} \shHO / F^k \shHO \bigr)^{\vee} \to 
		F^{n-k} \shHO / F^{n-k+1} \shHO,
\]
which, fiber by fiber, is given by $(2 \pi i)^{n-1}$ times integration over the smooth
hypersurface $\famX_p$ (up to a sign factor).

Since the duality theorem actually produces an exact triangle in the derived category
$D^b(P)$, we can take the push-forward to a point to arrive at the following global
statement. Note that $\Hom_P \bigl( \Gr_{n+1-k}^F \Mmod \bigr) = 0$ by
Theorem~\ref{thm:Ext-FkM}.

\begin{proposition} \label{prop:duality-global}
For each $k \in \ZZ$, we have an exact sequence
\begin{diagram}[l>=2em]
	H^0 \bigl( P, \Gr_k^F \Mmod \bigr) &\rIntoBold&
	H^0 \bigl( \famY, \psiu \OmX{n} \tensor \OfamY(k) \bigr) &\rOnto&
	\Ext_P^1 \bigl( \Gr_{n+1-k}^F \Mmod, \OP \bigr).
\end{diagram}
Moreover, we have $\Ext_P^{i+1} \bigl( \Gr_{n+1-k}^F \Mmod, \OP \bigr) \simeq H^i \bigl(
\famY, \psiu \OmX{n} \tensor \OfamY(k) \bigr)$ for all $i \geq 1$ and all values of
$k$.
\end{proposition}

\subsection{A curious vanishing theorem}

We digress to point out a curious application of the computations for the groups
$\Ext_P^i \bigl( \Gr_k^F \Mmod, \OP \bigr)$ when $i > 0$. We have obtained two different
expressions for these groups, one in Proposition~\ref{prop:duality-global}, the other
in Theorem~\ref{thm:Ext-FkM}. By comparing the two, we get the following statement
(which, I believe, is originally due to M.~Green).

\begin{proposition} \label{prop:curious-vanishing}
Let $X$ be a smooth projective variety of dimension $n$, and let $\OX(1)$ be a
sufficiently ample line bundle on $X$. For all $k \geq 0$ and $q > 0$, we have
\[
	H^q \bigl( X, \OmX{n} \tensor \Sym^k \NXQ \bigr) \simeq
	H^q \bigl( \famY, \psiu \OmX{n} \tensor \OfamY(k) \bigr) 
		\simeq \Hprim{n-k,q+k}(X),
\]	
where $\NXQ$ is the normal bundle for the embedding of $X$ into projective space.
\end{proposition}
\begin{proof}
Recall that $\psi \colon \famY \to X$ is the projective bundle $\PP \bigl( \NXQ
\bigr)$. For $k \geq 0$, we thus have $\psi_{\ast} \OfamY(k) \simeq \Sym^k \NXQ$. The
first asserted isomorphism is then obtained by push-forward along the map $\psi$.
The second isomorphism now follows immediately by combining the result of
Proposition~\ref{prop:duality-global} and Theorem~\ref{thm:Ext-FkM}.
\end{proof}

In particular, one gets a vanishing theorem for the ample vector bundle $\NXQ$. For
example, the condition $q + k \geq n + 1$ is sufficient to make the cohomology group
$H^q \bigl( X, \OmX{n} \tensor \Sym^k \NXQ \bigr)$ be zero. This does not seem to
follow from the standard vanishing theorems for ample vector bundles, such as
Griffiths' Theorem \cite{LazarsfeldII}*{Theorem~7.3.1 on p.~90}, because the factor
of $\det \NXQ$ is missing. At the same time, whenever $X$ has nontrivial primitive
cohomology in degree $(n-k,q+k)$, one gets an example where the group in question
is \emph{not} zero. This places restrictions on the kind of vanishing theorem one can
get for general ample vector bundles.

\subsection{The Serre conditions}
\label{subsec:Serre}

The sheaves $F_k \Mmod$ are natural extensions of the Hodge bundles $F^{n-k} \shHO$
on $\Psm$. While they cannot in general be locally free, we nevertheless expect that $F_k \Mmod$
should have good properties when the line bundle $\OX(1)$ is sufficiently
ample---basically because the complexity of the dual variety appears in very high
codimension. In this section, we illustrate this by proving that each $F_k \Mmod$ in
the range $1 \leq k \leq n$ is a reflexive sheaf.

More precisely, recall that a coherent sheaf $\shF$ on a nonsingular algebraic
variety $P$ is said to satisfy (a slight modification of) \define{Serre's condition
$S_m$}, if the inequality
\[
	\depth \shF_p \geq \min \bigl( m, \dim \shO_{P,p} \bigr)
\]
holds at every point $p \in P$. A locally free sheaf satisfies $S_m$ for all values
of $m$; on the other hand, condition $S_1$ is equivalent to being
torsion-free, and $S_2$ to being reflexive. There is also a useful
criterion for verifying Serre's condition: $\shF$ satisfies $S_m$ if and only if,
for every $i > 0$, the codimension of the support of $\shExt^i(\shF,
\OP)$ is at least $i+m$ (see \cite{Popa}*{Proposition~7.5} for a careful proof). In
particular, $\shF$ is then locally free in codimension $m$.

The following theorem shows that the sheaves $\Gr_k^F \Mmod$ and $F_k \Mmod$ are
well-behaved when $1 \leq k \leq n$, at least when $\OX(1)$ is sufficiently ample.

\begin{theorem} \label{thm:MSp}
Fix an integer $m \geq 1$. If the line bundle $\OX(1)$ is ample enough, then
each sheaf $\Gr_k^F \Mmod$ in the interval $1 \leq k \leq n$ satisfies Serre's
condition $S_m$.
\end{theorem}

The proof is based on bounding the codimension of the set of hyperplane sections with
many singular points, or with singular points of high multiplicity.

\begin{proposition} \label{prop:bound-sings}
Fix two positive integers $N$ and $r$. If $\OX(1)$ is sufficiently ample, then the
linear system $P$ has the following properties:
\begin{enumerate}[label=(\alph{*}), ref=(\alph{*})]
\item The subset $P_1(N) \subseteq P$ of hypersurfaces with at least $N$ singular points
has codimension at least $N$ in $P$. \label{en:sing1}
\item The subset $P_2(r) \subseteq P$ of hypersurfaces with a singular point of
multiplicity at least $r$ has codimension at least $r-1$ in $P$. \label{en:sing2}
\end{enumerate}
\end{proposition}

\begin{proof}
As usual, let $d = \dim P$. 
To prove \ref{en:sing1}, let $\famS_1(N) \subseteq P \times X^N$ be the closure of
the set of points $(p, x_1, \dotsc, x_N)$ for which $x_1, \dotsc, x_N$ are distinct
singular points of the hypersurface $\famX_p$. Then $\famS_1(N)$ is irreducible, and
its fiber over a general point $(x_1, \dotsc, x_N) \in X^N$ has dimension $d -
N(n+1)$ by Lemma~\ref{lem:jet-sep}, provided that $\OX(1)$ is sufficiently ample. Therefore $\dim
\famS_1(N) = d-N$, and so its projection to $P$, which equals $P_1(N)$, has dimension
at most $d - N$.

The argument for \ref{en:sing2} is similar; this time, consider the set $\famS_2(r)
\subseteq P \times X$ of points $(p,x)$ such that $x \in \famX_p$ is a singular point
of multiplicity at least $r$. By Lemma~\ref{lem:jet-sep}, we can take $\OX(1)$
sufficiently ample to separate $r$-jets; then the fiber of $\famS_2(r)$ over a point
$x \in X$ has dimension at most $d - \binom{n+r-1}{r-1}$, and so the dimension of
$P_2(r)$ cannot be greater than
\[
	n + d - \binom{n+r-1}{r-1} \leq d-r+1. \qedhere
\]
\end{proof}

During the proof, we used the following lemma about separation of jets; it is a straightforward
generalization of \cite{LazarsfeldI}*{Theorem~5.1.17}.

\begin{lemma} \label{lem:jet-sep}
Fix two positive integers $N$ and $r$. If $\OX(1)$ is sufficiently ample, then for any
collection of $N$ distinct points $x_1, \dotsc, x_N \in X$, the restriction map
\begin{equation} \label{eq:jet-sep1}
	H^0 \bigl( X, \OX(1) \bigr) \to \bigoplus_{i=1}^N \OX(1) \tensor 
		\shO_{X, x_i} / \mathfrak{m}_{x_i}^{r+1}
\end{equation}
is surjective. 
\end{lemma}

We now turn to the proof of Theorem~\ref{thm:MSp}.
\begin{proof}
It suffices to show that the sheaves $\shExt^i \bigl( \Gr_{n+1-k}^F \Mmod, \OP \bigr)$
are supported in codimension at least $i+m$, for all $i > 0$.  First, we treat the
case when $i \geq 2$, where we have
\[
	\shExt^i \bigl( \Gr_{n+1-k}^F \Mmod, \OP \bigr) \simeq
		R^{i-1} \pil \bigl( \psiu \OmX{n} \tensor \OfamY(k) \bigr)
\]
by \eqref{eq:dual2}. The support of this sheaf is therefore contained in the locus
$P_1$ where the singular set of the fibers has positive
dimension; by Proposition~\ref{prop:bound-sings}, we can make its codimension greater
than $i+m$ by taking $\OX(1)$ sufficiently ample.

What remains is the case $i = 1$. The exact sequence in \eqref{eq:dual1}
shows that $\shExt^1 \bigl( \Gr_{n+1-k} \Mmod, \OP \bigr)$ is the cokernel of the
restriction map
\[
	\Gr_k^F \Mmod \to \pil \bigl( \psiu \OmX{n} \tensor \OfamY(k) \bigr).
\]
To complete the proof, we have to show that this map is surjective except on a set of
codimension at least $m+1$. We may restrict our attention to the open subset $P
\setminus P_1$ where $\pi \colon \famY \to P$ has finite fibers; over that set, $\pi$
is a finite morphism, and so the stalk of $\pil \bigl( \psiu \OmX{n} \tensor
\OfamY(k) \bigr)$ at a point $p \in P \setminus P_1$ equals $H^0 \bigl( \famY_p,
\OmX{n}(k) \bigr)$. Since we know from Corollary~\ref{cor:glob-sect} that
$H^0 \bigl( X, \OmX{n}(k) \bigr) \tensor \OP(k) \to \Gr_k^F \Mmod$
is always surjective, it is therefore sufficient to prove that the map
\begin{equation} \label{eq:sing}
	H^0 \bigl( X, \OmX{n}(k) \bigr) \to H^0 \bigl( \famY_p, \OmX{n}(k) \bigr)
\end{equation}
is surjective, except when $p \in P \setminus P_1$ lies in a subset of codimension
$\geq m+1$.

Let $P_2 \subseteq P$ be the set of hyperplane sections $\famX_p$ that either have
more than $m$ singular points, or have a singular point of multiplicity greater
than $m+1$; by Proposition~\ref{prop:bound-sings}, the codimension of $P_2$ is at
least $m+1$. On the other hand, \eqref{eq:sing} is surjective for each $p \in P
\setminus P_2$; indeed, $\famY_p$ consists of at most $m$ points, each with
multiplicity no greater than $m+1$, and so we can apply Lemma~\ref{lem:jet-sep}. This
concludes the proof that $\Gr_k^F \Mmod$ satisfies Serre's condition $S_m$.
\end{proof}

\begin{corollary}
If $\OX(1)$ is sufficiently ample, then for each $k=1, \dotsc, n$, the sheaves
$\Gr_k^F \Mmod$ and $F_k \Mmod$ are reflexive.
\end{corollary}

\end{document}